\documentclass[11pt]{amsart}
\usepackage{setspace}
\usepackage{graphicx}
\usepackage{physics}
\usepackage{geometry}
\usepackage{multirow}
\usepackage{mdframed}
\usepackage{amsmath, amsthm}
\usepackage{thmtools}
\usepackage{bbm}
\usepackage{thm-restate}

\usepackage{amssymb}
\usepackage{array}
\newcolumntype{M}[1]{>{\centering\arraybackslash}m{#1}}
\usepackage{mathtools}

\newtagform{roman}[]()
\usepackage[utf8]{inputenc}
\theoremstyle{definition}
\usepackage{amsthm}
\usepackage{gensymb}
\usepackage{xfrac}
\usepackage{commath}
\usepackage{float}
\usepackage{caption}
\usepackage{color}
\usepackage{enumerate}
\usepackage{hyperref}
\hypersetup{
   colorlinks = true, 
   linkcolor = black, 
   filecolor = black,      
   urlcolor = black, 
   citecolor = black
}
\usepackage{url}
\expandafter\def\expandafter\UrlBreaks\expandafter{\UrlBreaks
  \do\a\do\b\do\c\do\d\do\e\do\f\do\g\do\h\do\i\do\j
  \do\k\do\l\do\m\do\n\do\o\do\p\do\q\do\r\do\s\do\t
  \do\u\do\v\do\w\do\x\do\y\do\z\do\A\do\B\do\C\do\D
  \do\E\do\F\do\G\do\H\do\I\do\J\do\K\do\L\do\M\do\N
  \do\O\do\P\do\Q\do\R\do\S\do\T\do\U\do\V\do\W\do\X
  \do\Y\do\Z}
\usepackage{setspace}
\usepackage{systeme}
\usepackage{microtype}
\makeatletter
\newsavebox\myboxA
\newsavebox\myboxB
\newlength\mylenA
\newcommand*\xoverline[2][1]{%
    \sbox{\myboxA}{$\m@th#2$}%
    \setbox\myboxB\null
    \ht\myboxB = \ht\myboxA%
    \dp\myboxB = \dp\myboxA%
    \wd\myboxB = #1\wd\myboxA
    \sbox\myboxB{$\m@th\overline{\copy\myboxB}$}
    \setlength\mylenA{\the\wd\myboxA}
    \addtolength\mylenA{-\the\wd\myboxB}%
    \ifdim\wd\myboxB<\wd\myboxA%
       \rlap{\hskip 0.5\mylenA\usebox\myboxB}{\usebox\myboxA}%
    \else
        \hskip -0.5\mylenA\rlap{\usebox\myboxA}{\hskip 0.5\mylenA\usebox\myboxB}%
    \fi}
\makeatother
\usepackage[notquote, notperiod, notquery, notexcl, notcolon, notscolon, notcomma]{hanging}
\usepackage{color}

\usepackage[normalem]{ulem}
\usepackage{xcolor}
\makeatletter
\def\squiggly{\bgroup \markoverwith{\textcolor{black}{\lower3.5\p@\hbox{\sixly \char58}}}\ULon}
\makeatother
\usepackage{xr}
\usepackage{calc}
\usepackage{tkz-euclide}
\usepackage{diagbox}
\usepackage{tabularx}
\usepackage{cite}
\usepackage{etoolbox}
\usepackage[british]{babel}
\theoremstyle{theorem}
\newtheorem{theorem}{Theorem}[section]
\newtheorem{lemma}[theorem]{Lemma}
\newtheorem{cor}[theorem]{Corollary}

\newtheorem{question}[theorem]{Question}
\newtheorem{problem}[theorem]{Problem}

\newtheorem{prop}[theorem]{Proposition}

\newtheorem*{claim*}{Claim}
\newtheorem*{theorem*}{Theorem}
\newtheorem*{prop*}{Proposition}
\newtheorem*{lemma*}{Lemma}
\newtheorem*{maintheorem*}{Main Theorem}
\newtheorem*{conjecture*}{Conjecture}
\numberwithin{equation}{section}
\theoremstyle{definition}
\newtheorem{defn}[theorem]{Definition}
\newtheorem{example}[theorem]{Example}
\newtheorem{remark}[theorem]{Remark}

\newcommand{\twoN}[0]{2^{\mathbb{N}}}

\newcommand{\R}[0]{\mathbb{R}}
\newcommand{\Q}[0]{\mathbb{Q}}
\newcommand{\N}[0]{\mathbb{N}}
\newcommand{\Z}[0]{\mathbb{Z}}
\newcommand{\M}[0]{\mathcal{M}}

\counterwithout{equation}{section}
\title[Distal expansions of Presburger arithmetic by a sparse predicate]{Distal expansions of Presburger arithmetic\\by a sparse predicate}
\author{Mervyn Tong}
\address{School of Mathematics, University of Leeds, Leeds LS2 9JT, United Kingdom}
\email{mmhwmt@leeds.ac.uk}
\date{\today}
\begin{document}
\maketitle
\begin{abstract}
    We prove that the structure $(\Z,<,+,R)$ is distal for all congruence-periodic sparse predicates $R\subseteq\N$. We do so by constructing a strong honest definition for every formula $\phi(x;y)$ with $\abs{x}=1$, providing a rare example of concrete distal decompositions.
\end{abstract}
\section{Introduction}\label{sectionintro}
One of the most important threads of model-theoretic research is identifying and studying \textit{dividing lines} in the universe of structures: properties $\mathcal{P}$ such that structures with $\mathcal{P}$ are `tame' and `well-behaved' in some sense.

Two dividing lines that have attracted much interest, not just in model theory but also in fields such as combinatorics and machine learning, are \textit{stability} and \textit{NIP} (`not the independence property'). In \cite{simondistal}, Simon introduced the dividing line of `distality', intended to characterise NIP structures that are `purely unstable'. Indeed, stability and distality can be viewed as two opposite ends of the NIP spectrum: no infinite structure satisfies both simultaneously. However, a stable structure can admit a distal expansion, and this is (a special case of) the subject of curiosity among many model theorists, phrased in \cite{distalityvaluedfields} as the following question.
\begin{question}\label{distalexpansion}
    Which NIP structures admit distal expansions?
\end{question}
The reason (or one such reason) this is a question of interest is precisely the fact that distal structures have nice structural properties. Most notably, in \cite{distaldefn}, Chernikov and Simon prove that a structure $\mathcal{M}$ is distal if and only if every every formula $\phi(x;y)$ in its theory has a \textit{strong honest definition}, or (as termed in \cite{galvin}) a \textit{distal (cell) decomposition}. Informally, this means that given a finite set $B\subseteq M^y$, there is a decomposition of $M^x$, uniformly definable from $B$, into finitely many cells, such that the truth value of $\phi(x;b)$ is constant on each cell for all $b\in B$.

Cell decompositions in general have proved useful for deriving various results, particularly of a combinatorial nature, and distal decompositions are no exception. Many results that hold in the real field, where we have semialgebraic cell decomposition, that were found to generalise to o-minimal structures, where we have o-minimal cell decomposition, turn out to also generalise to distal structures, where we have distal decomposition (note that o-minimal structures are distal, and in fact, o-minimal cell decomposition is a special case of distal decomposition). A notable example concerns the strong Erd\H{o}s--Hajnal property. It was shown in \cite{realerdoshajnal} that every definable relation over the real field has the strong Erd\H{o}s--Hajnal property. This was later generalised in \cite{ominerdoshajnal} to every definable, topologically closed relation in any o-minimal expansion of a real-closed field. Finally, it was shown in \cite{regularitylemma} that a structure is distal \textit{if and only if} every relation in its theory satisfies the definable strong Erd\H{o}s--Hajnal property.

Such results support the view (such as in \cite{galvin}) that distality is the `correct' context in which combinatorics should be done; in other words, that the distal decomposition is the `correct' strength of decomposition that affords structures good combinatorial properties.

The main result of this paper thus fits nicely into the context described above.
\begin{maintheorem*}[Theorem \ref{final}]
    Let $R\subseteq \N$ be a congruence-periodic sparse predicate. Then the structure $(\Z,<,+,R)$ is distal.
\end{maintheorem*}
Note that, by \cite[Corollary~2.20]{dolichgoodrick}, such structures $(\Z,<,+,R)$ have dp-rank $\geq\omega$, so our main theorem completely classifies these structures on the model-theoretic map of the universe.

Here, \textit{congruence-periodic} means that, for all $m\in\N^+$, the increasing sequence by which $R$ is enumerated is eventually periodic modulo $m$. \textit{Sparsity} will be defined in Definition \ref{sparsedefn}, but for now we content ourselves by noting that sparse predicates include such examples as $d^\N:=\{d^n: n\in\N\}$ for any $d\in\N_{\geq 2}$, the set of Fibonacci numbers, and $\{n!: n\in\N\}$.

We now give an overview of how this result extends and builds on results in the extant literature. In \cite{regularnip}, Lambotte and Point prove that $(\Z,+,<,R)$ is NIP for all congruence-periodic sparse predicates $R\subseteq\N$, so our result is a strengthening of theirs. They also define the notion of a \textit{regular} predicate, show that regular predicates are sparse, allowing them to apply their result to congruence-periodic regular predicates. It turns out that the converse holds: sparse predicates are regular, which we prove in Theorem \ref{sparseisregular} as a result of independent interest, providing an equivalent, more intuitive definition of sparsity.

In the same paper, they also prove that $(\Z,+,R)$ is superstable for all regular predicates $R\subseteq\N$. So, if additionally $R$ is congruence-periodic, then our result shows that $(\Z,+,R)$ admits a distal expansion, namely, $(\Z,+,<,R)$. This provides a large class of examples of stable structures with distal expansions, which should provide intuition towards an answer to Question \ref{distalexpansion}. We note that examples of NIP structures \textit{without} distal expansions are far more meagre, and so far the only known method of proving that a structure does not have a distal expansion is to exhibit a formula without the strong Erd\H{o}s--Hajnal property (see \cite{regularitylemma}). It is our hope that our more direct proof of distality may provide new methods and insights to that end.

To our knowledge, no examples of $R\subseteq \N$ are known such that $(\Z,<,+,R)$ is NIP but not distal. As discussed above, distality is a desirable strengthening of NIP, so it would be pleasant if NIP sufficed for distality for such structures. We therefore ask the following question.

\begin{question}
    Is there $R\subseteq\N$ such that $(\Z,<,+,R)$ is NIP but not distal?
\end{question}
In fact, even the existence of a non-distal NIP expansion of $(\N,<)$ appears to be unknown --- see \cite[Question~11.16]{walsberg}. More broadly, we would like to understand the following problem.
\begin{problem}
    Characterise the class of predicates $R\subseteq \N$ such that $(\Z,<,+,R)$ is distal.
\end{problem}
A natural first step to understanding this problem is to ask the following question.

\begin{question}
    Let $R\subseteq\N$ be sparse but not necessarily congruence-periodic. Must the structure $(\Z,<,+,R)$ be distal?
\end{question}

Congruence-periodicity is used in an essential way in our proof, so we expect that a substantial change in approach would be required to provide a positive answer to this question.

We had previously wondered if every non-distal structure of the form $(\Z,<,+,R)$ interprets arithmetic, but $R=2^\N \cup 3^\N$ serves as a counterexample\footnote{We thank Gabriel Conant for bringing this to our attention.}. Indeed, the resulting structure does not interpret arithmetic \cite{schulz} and is IP (hence non-distal). A proof of the latter is given in \cite{hieronymischulz} (where, in fact, 2-IP is claimed), but in personal communication with the authors an error was found; they have nonetheless supplied an alternative argument that the structure is (1-)IP.

Our original motivation for proving the main theorem was to answer a question of Michael Benedikt (personal communication), who asked whether the structure $(\Z,<,+,\twoN)$ was distal. His motivation was to know whether the structure has so-called \textit{Restricted Quantifier Collapse (RQC)}, a property satisfied by all distal structures \cite{benedikt}. In personal communication, he informed us that he is also interested in obtaining better VC bounds for formulas in this structure (coauthoring \cite{benediktvc} to that end), and that a constructive proof of distality could help in this endeavour. Our proof is nothing but constructive.
\subsection*{Strategy of our proof and structure of the paper}
The proof of our main theorem, Theorem \ref{final}, comprises most of the paper. In Section \ref{sectionprelim}, we define and motivate the terminology used in our main theorem, and state and prove basic facts about sparse predicates that are either useful for our proof or of independent interest. Our proof begins in earnest in Section \ref{sectionreduction}.

Let us describe the strategy of the proof. Perhaps its most noteworthy feature, and what distinguishes it from most other proofs of distality, is that we prove that the structure is distal by giving explicit strong honest definitions (hence, distal decompositions) for `representative' formulas of the theory. Most proofs of distality in the literature go via the original definition of distality (given in \cite{simondistal}) using indiscernible sequences, which offers no information on the structure or complexity (such as `distal density') of the distal decomposition, which is itself a subject of interest, such as in \cite{anderson}. As phrased in \cite{distalityvaluedfields}, `occasionally [the characterisation of distality via strong honest definitions] is more useful since it ultimately gives more information about definable sets, and obtaining bounds on the complexity of strong honest definitions is important for combinatorial applications'.

The first stage of the proof is thus to characterise `representative' formulas of the theory, which is the goal of Section \ref{sectionreduction}. The main result in that section is Theorem \ref{sufficiency}, where we show that to prove the distality of our structure, it suffices to construct strong honest definitions for suitable so-called $(F_n; ...)$ formulas (where $n\in\N^+$), to be defined in Definition \ref{defnformulas}. We prove this by first showing that every formula $\phi(x;y)$ with $\abs{x}=1$ is (essentially) equivalent to a Boolean combination of so-called $(E_n)$ formulas (Proposition \ref{agenda}), and then showing that every $(E_n)$ formula is (essentially) equivalent to a Boolean combination of suitable $(F_n; ...)$ formulas and $(E_{n-1})$ formulas (Corollary \ref{EntoFn}). By induction on $n\in\N^+$, this gives an explicit recipe for writing every formula $\phi(x;y)$ with $\abs{x}=1$ as (essentially) a Boolean combination of suitable $(F_n; ...)$ formulas. This is summarised precisely at the end of Section \ref{sectionreduction}.

%
Constructing strong honest definitions for $(F_n; ...)$ formulas is the goal of Section \ref{sectionmain} of the paper. The broad strategy is to induct on $n\in\N^+$. Theorem \ref{catchall}, which produces new strong honest definitions from existing ones, is a stronger version of the base case $n=1$ (Corollary \ref{onedimension}), and is also a key ingredient in the inductive step (Theorem \ref{multidimension}). Morally, the base case is $n=0$ (see Corollary \ref{onedimension}), where the formula is a formula of Presburger arithmetic, hence admitting a strong honest definition since Presburger arithmetic is distal; Corollary \ref{onedimension} bootstraps this strong honest definition to construct ones for $(F_1; ...)$ formulas using Theorem \ref{catchall}. Thus, the proof strategy can be described as `generating strong honest definitions in $(\Z,<,+,R)$ from ones in the distal structure $(\Z,<,+)$', which may prove a useful viewpoint for similar applications in the future.

We thus give a recipe to construct explicit strong honest definitions, and thus distal decompositions, for all formulas $\phi(x;y)$ with $\abs{x}=1$. However, we make no comment on the structure of these distal decompositions, as the complexity of our construction renders such analysis a separate project. In particular, we make no claim on the `optimality' of our decomposition, to which little credence is lent by the length of our construction anyway. The objective of this paper is to provide a rare example of concrete distal decompositions, which the reader may analyse for aspects of distal decompositions in which they are interested.
\subsection*{Acknowledgements}
We thank Pantelis Eleftheriou for providing numerous helpful suggestions on the content and structure of this paper, as well as Pablo And\'{u}jar Guerrero and Aris Papadopoulos for fruitful discussions on distality. We would also like to thank the referee for their helpful comments and corrections. \textit{Soli Deo gloria.}
\section{Preliminaries and basic facts}\label{sectionprelim}
This section lays out the two key definitions in our main theorem --- distality of a structure and sparsity of a predicate --- and provides some commentary on these notions.
\subsection{Distality}
Let us begin by defining distality. As mentioned in the introduction, distality was originally defined by Simon in \cite{simondistal} using indiscernible sequences, but we shall take the following --- proven to be equivalent by Chernikov and Simon in \cite[Theorem~21]{distaldefn} --- as our definition of distality. Recall that if $\phi(x;y)$ is a formula in a structure $\mathcal{M}$, $a\in M^x$, and $B\subseteq M^y$, the \textit{$\phi$-type of $a$ over $B$} is $\text{tp}_\phi(a/B):=\{\phi(x;b): b\in B, \mathcal{M}\models \phi(a;b)\}\cup \{\neg\phi(x;b): b\in B, \mathcal{M}\models \neg\phi(a;b)\}$.
\begin{defn}\label{distaldefn}
    Say that an $\mathcal{L}$-structure $\mathcal{M}$ is \textit{distal} if for every partitioned $\mathcal{L}$-formula $\phi(x;y)$, there is a formula $\psi(x;y^{(1)},...,y^{(k)})$ such that for all $a\in M^x$ and finite $B\subseteq M^y$ with $\abs{B}\geq 2$, there is $c\in B^k$ such that $a\models \psi(x;c)$ and $\psi(x;c)\vdash \text{tp}_\phi(a/B)$, that is, for all $a'\models \psi(x;c)$ and $b\in B$, $\mathcal{M}\models \phi(a;b)\leftrightarrow \phi(a';b)$.
\end{defn}
In Definition \ref{distaldefn}, $\psi$ is known as a \textit{strong honest definition} (in $\mathcal{M}$) for $\phi$. By \cite[Proposition~1.9]{distalityvaluedfields}, when showing that $\mathcal{M}$ is distal, it suffices to verify that every partitioned formula $\phi(x;y)$ with $\abs{x}=1$ has a strong honest definition.

The following lemma is straightforward to prove.
\begin{lemma}\label{boolcomb}
    Let $\phi_1(x;y)$ and $\phi_2(x;y)$ be formulas, respectively with strong honest definitions $\psi_1(x;y^{(1)}, ..., y^{(k)})$ and $\psi_2(x;y^{(1)}, ..., y^{(l)})$.
    \begin{enumerate}[(i)]
        \item The formula $\neg\phi_1(x;y)$ has strong honest definition $\psi_1(x;y^{(1)}, ..., y^{(k)})$.
        \item The formula $\phi_1\wedge \phi_2(x;y)$ has strong honest definition 
        \[\psi_1(x;y^{(1)}, ..., y^{(k)})\wedge\psi_2(x;y^{(k+1)}, ..., y^{(k+l)}).\]
    \end{enumerate}
\end{lemma}
When constructing a strong honest definition for $\phi(x;y)$, it is often convenient to partition $M^x$ into finitely many pieces and use a different formula for each piece. This is the content of the following proposition.
\begin{prop}\label{system}
    Fix an $\mathcal{L}$-structure $\mathcal{M}$ and a partitioned $\mathcal{L}$-formula $\phi(x;y)$. Then $\phi$ has a strong honest definition in $\mathcal{M}$ if and only if there is a finite set of formulas $\Psi(x;y^{(1)},...,y^{(k)})$ such that for all $a\in M^x$ and finite sets $B\subseteq M^y$ with $\abs{B}\geq 2$, there is $c\in B^k$ and $\psi\in \Psi$ such that $a\models \psi(x;c)$ and $\psi(x;c)\vdash \mathrm{tp}_\phi(a/B)$.
\end{prop}
\begin{proof}
    The forward direction is immediate. Let $\Psi(x;y^{(1)},...,y^{(k)})$ witness the antecedent of the backward direction; enumerate its elements as $\psi_1,...,\psi_n$. Let
    \[\theta(x;y^{(i,1)}, ..., y^{(i,k)},u^{(i)},v^{(i)}: 1\leq i\leq n):=\bigvee_{i=1}^n \left(u^{(i)}=v^{(i)}\wedge \psi_i(x;y^{(i,1)},...,y^{(i,k)})\right),\]
    where $u^{(i)}, v^{(i)}$ are tuples of variables of length $\abs{y}$. We claim that this is a strong honest definition for $\phi$. Fix $a\in M^x$ and finite $B\subseteq M^y$ with $\abs{B}\geq 2$. There is $c\in B^k$ and $1\leq j\leq n$ such that $a\models \psi_j(x;c)$ and $\psi_j(x;c)\vdash \text{tp}_\phi(a/B)$. Choose $u^{(1)},v^{(1)}, ..., u^{(n)},v^{(n)}\in B$ such that $u^{(i)}=v^{(i)}$ if and only if $i=j$; this is possible since $\abs{B}\geq 2$. Then $a\models \theta(x;c, u^{(i)}, v^{(i)}: 1\leq i\leq n)$ since $a\models \psi_j(x;c)$, and $\theta(x;c, u^{(i)}, v^{(i)}: 1\leq i\leq n)\vdash \psi_j(x;c)$ since $u^{(j)}\neq v^{(j)}$ for all $j\neq i$. But now $\psi_j(x;c)\vdash \text{tp}_\phi(a/B)$.
\end{proof}
Call such a $\Psi$ a \textit{system of strong honest definitions} (in $\mathcal{M}$) for $\phi$.
\begin{defn}
    Let $\phi(x; y)$ be an $\mathcal{L}$-formula with $m:=\abs{x}$ and $n:=\abs{y}$. Say that an $\mathcal{L}$-formula $\theta(u; v)$ is a \textit{descendant} of $\phi$ if
    \[\theta(u; v)=\phi(f_1(u), ..., f_m(u); g_1(v), ..., g_n(v))\]
    for some $\mathcal{L}$-definable functions $f_1, ..., f_m$ of arity $\abs{u}$ and $g_1, ..., g_n$ of arity $\abs{v}$.
\end{defn}
Note that the descendant relation is reflexive and transitive.
\begin{lemma}\label{descendants}
    Fix an $\mathcal{L}$-structure $\mathcal{M}$ with at least two $\emptyset$-definable elements. If an $\mathcal{L}$-formula $\phi(x;y)$ has a strong honest definition, so does any descendant of $\phi$.
\end{lemma}
\begin{proof}
    Let $\alpha,\beta\in M$ be distinct $\emptyset$-definable elements. Let $\phi(x; y)$ be an $\mathcal{L}$-formula with $m:=\abs{x}$ and $n:=\abs{y}$, and suppose it has a strong honest definition $\psi(x;y^{(1)},...,y^{(k)})$. Let
    \[\theta(u; v)=\phi(f_1(u), ..., f_m(u); g_1(v), ..., g_n(v))\]
    be a descendant of $\phi$, for some $\mathcal{L}$-definable functions $f_1, ..., f_m$ and $g_1, ..., g_n$.
    
    Write $[k]:=\{1, ..., k\}$. For $I\sqcup J\subseteq [k]$ (that is, $I,J\subseteq [k]$ disjoint), let
    \[\zeta_{IJ}(u;v^{(i)}: i\in [k]\setminus(I\cup J)):=\psi(f_1(u), ..., f_m(u); h^{(1)}_1, ..., h^{(1)}_n, ..., h^{(k)}_1, ..., h^{(k)}_n),\]
    where
    \[h^{(i)}_j=\begin{cases}
        \alpha&\text{if }i\in I,\\
        \beta&\text{if }i\in J,\\
        g_j(v^{(i)})&\text{otherwise}.
    \end{cases}\]
    We claim that $\{\zeta_{IJ}: I\sqcup J\subseteq [k]\}$ is a system of strong honest definitions for $\theta$. Indeed, let $a\in M^u$ and $B\subseteq M^v$ with $2\leq \abs{B}<\infty$. Let $\bar{B}:=\{(g_1(v), ..., g_n(v)): v\in B\}$, and let
    \[\hat{B}:=\bar{B}\cup\{(\alpha, ..., \alpha), (\beta, ..., \beta)\}\subseteq M^n.\]
    Since $\psi$ is a strong honest definition for $\phi$ and $2\leq|\hat{B}|<\infty$, there is $c=(c^{(1)}, ..., c^{(k)})\in \hat{B}^k$ such that $(f_1(a), ..., f_m(a))\models \psi(x;c)$ and $\psi(x;c)\vdash \text{tp}_\phi (f_1(a), ..., f_m(a)/\hat{B})\supseteq\text{tp}_\phi (f_1(a), ..., f_m(a)/\bar{B})$.
    
    Let $I:=\{i\in [k]: c^{(i)}=(\alpha, ..., \alpha)\}$ and $J:=\{i\in [k]: c^{(i)}=(\beta, ..., \beta)\}$. Then, there is a tuple $(w^{(i)}: i\in [k]\setminus (I\cup J))$ from $B$ such that
    \[\psi(f_1(u), ..., f_m(u);c)=\zeta_{IJ}(u;w^{(i)}: i\in [k]\setminus(I\cup J)),\]
    whence $a\models \zeta_{IJ}(u;w^{(i)}: i\in [k]\setminus(I\cup J))$ and $\zeta_{IJ}(u;w^{(i)}: i\in [k]\setminus(I\cup J))\vdash \text{tp}_\theta(a/B)$.
\end{proof}
\begin{remark}\label{remarkdescendants}
    In the proof above, if the function $v\mapsto (g_1(v), ..., g_n(v))$ were injective, then
    \[\zeta(u;v^{(i)}: i\in [k]):=\psi(f_1(u), ..., f_m(u); g_1(v^{(1)}), ..., g_n(v^{(1)}), ..., g_1(v^{(k)}), ..., g_n(v^{(k)}))\]
    would have sufficed as a strong honest definition for $\theta$.
\end{remark}
\begin{example}\label{LOexample}
    It is well-known that Presburger arithmetic is distal (see, for example, \cite[Example~2.9]{regularitylemma}), but, as an example, let us prove this by constructing strong honest definitions. Another well-known fact about Presburger arithmetic (see, for example, \cite{cooper}) is that it admits quantifier elimination in the language $\mathcal{L}_{\text{Pres}}:=(<,+,-,0,1,(m\mid\cdot)_{m\in\N^+})$, where $m\mid\cdot$ is a unary relation symbol interpreted as divisibility by $m$.
    
    It thus suffices to construct a strong honest definition for every atomic $\mathcal{L}_{\text{Pres}}$-formula $\phi(x;y)$ with $\abs{x}=1$. These have the form $f(x,y)=0$, $f(x,y)<0$, or $m\mid f(x,y)$, where $f$ is a $\Z$-affine function. We can ignore formulas of the form $f(x,y)=0$, since $f(x,y)=0\leftrightarrow f(x,y)<1\wedge -f(x,y)<1$. By Lemma \ref{descendants}, it suffices to construct strong honest definitions for $\phi(x;y):=x<y$ and $\psi_m(x;y):=m\mid (x-y)$.
    
    The formula $\phi(x;y)$ admits a system of strong honest definitions given by $\{x<y, x=y, y<x, y<x<y'\}$, where $\abs{y'}=\abs{y}$; in what follows, we will understand $-\infty<x<y$ to mean $x<y$ and $y<x<+\infty$ to mean $y<x$. Indeed, let $a\in\Z$ and $B\subseteq \Z$ with $2\leq \abs{B}<\infty$. Enumerate $B$ as $\{b_1, ..., b_n\}$, where $b_1< \cdots < b_n$. If there is $1\leq i\leq n$ such that $a=b_i$, then $a\models x=b_i$ and $x=b_i\vdash \text{tp}_{\phi}(a/B)$. Otherwise, there is $0\leq i\leq n$ such that $b_i<a<b_{i+1}$ (where $b_0:=-\infty$ and $b_{n+1}:=+\infty$), whence $a\models b_i<x<b_{i+1}$ and $b_i<x<b_{i+1}\vdash \text{tp}_{\phi}(a/B)$.

    The formula $\psi_m(x;y)$ admits a system of strong honest definitions given by $\{m\mid (x-i): 0\leq i<m\}$. Indeed, let $B\subseteq \Z$ with $2\leq \abs{B}<\infty$. Given $a\in\Z$, there is $0\leq i<m$ such that $a\equiv i\mod m$, whence $a\models m\mid (x-i)$ and $m\mid (x-i)\vdash \text{tp}_{\psi_m}(a/\Z)$.

\end{example}
\begin{remark}
    Recall that, in the introduction, we discussed that every formula $\phi(x;y)$ in the theory of a distal structure $\mathcal{M}$ has a \textit{distal (cell) decomposition}. In spirit, this is the same as a strong honest definition, but its characterisation in terms of a partition of $M^x$ provides more concrete structural and combinatorial information. For this reason, compared to strong honest definitions, distal decompositions arguably give a more attractive characterisation of distal structures. However, in this paper we will continue to talk about constructing strong honest definitions, as this is the cleaner definition to work with; the reader should nonetheless keep in mind the implications regarding distal decompositions, and is referred to \cite{galvin} for a more detailed account of distal decompositions.
\end{remark}
\subsection{Sparsity}
Let us now define sparse predicates; these were introduced by Semenov in \cite{semenov}. For an infinite predicate $R\subseteq\N$ enumerated by the increasing sequence $(r_n:n\in\N)$, let $\sigma: R\to R$ denote the successor function, that is, $\sigma(r_n)=r_{n+1}$ for all $n\in \N$. By an \textit{operator} on $R$ we mean a function $R\to\Z$ of the form $a_n \sigma^n +\cdots+a_0 \sigma^0$, where $a_n, ..., a_0\in\Z$ and $\sigma^0$ is the identity function. For operators $A$ and $B$, write
\[\begin{cases}
    A=_R B&\text{if }Az=Bz\text{ for all }z\in R,\\
    A>_R B&\text{if }Az>Bz\text{ for cofinitely many }z\in R,\\
    A<_R B&\text{if }Az<Bz\text{ for cofinitely many }z\in R.
\end{cases}\]
The subscript $R$ is dropped where obvious from context. We also use $\sigma^{-1}$ to denote the predecessor function, where by convention we define $\sigma^{-1} (\min R):=\min R$.
\begin{example}\label{dNoperator}
Consider the predicate $d^\N := \{d^n: n\in \N\}$ for some fixed $d\in\N_{\geq 2}$, and let $A$ be an operator on $d^\N$, say of the form $a_n\sigma^n+\cdots+a_0\sigma^0$ where $a_n, ..., a_0\in\Z$. Then, for all $z\in d^\N$ we have $Az=(a_nd^n+\cdots + a_0d^0)z$, so the action of $A$ on $d^\N$ is multiplication by the constant $a_nd^n+\cdots + a_0d^0$.
\end{example}
\patchcmd{\thmhead}{(#3)}{#3}{}{}
\begin{defn}[{\cite[\S3]{semenov}}]\label{sparsedefn}
    Say that an infinite predicate $R\subseteq\N$ is \textit{sparse} if every operator $A$ on $R$ satisfies the following:
    \begin{enumerate}
        \item [(S1)] $A=_R 0$, $A>_R 0$, or $A<_R 0$; and
        \item [(S2)] If $A>_R 0$, then there exists $\Delta\in\N$ such that $A\sigma^\Delta z>z$ for all $z\in R$.\label{sparse2}
    \end{enumerate}
\end{defn}
\patchcmd{\thmhead}{#3}{(#3)}{}{}
\begin{example}\label{dNexample}
    Consider again the predicate $d^\N = \{d^n: n\in \N\}$ for some fixed $d\in\N_{\geq 2}$. By Example \ref{dNoperator}, every operator $A$ on $d^\N$ acts as multiplication by a constant $\lambda_A\in\Z$. Thus, (S1) is clearly satisfied. Furthermore, $A>_R 0$ if and only if $\lambda_A>0$, in which case $A\sigma z= \lambda_A dz>z$ for all $z\in d^\N$, so (S2) is also satisfied and $d^\N$ is sparse.
\end{example}
Other examples of sparse predicates, given by Semenov in \cite[\S3]{semenov}, include the set of Fibonacci numbers, $\{n!: n\in\N\}$, and $\{\lfloor e^n\rfloor: n\in\N\}$.

On the other hand, for all $f\in\N[x]$, the predicate $f(\N)=\{f(n): n\in\N\}$ is not sparse. Indeed, let $f\in\N[x]$; assume without loss of generality that $\deg f\geq 1$. Let $A$ be the operator $\sigma^1-\sigma^0$, so $A>_R 0$ since $f$ is strictly increasing. There is $g\in\N[x]$ with $\deg g<\deg f$ such that $Af(n)=f(n+1)-f(n)=g(n)$ for all $n\in\N$. Hence, for all $\Delta\in\N$, $A\sigma^\Delta f(n)=Af(n+\Delta)=g(n+\Delta)<f(n)$ for sufficiently large $n\in\N$.
\begin{remark}
    It may be tempting to conjecture from these examples and non-examples that $R=(r_n:n\in\N)\subseteq\N$ is sparse if and only if $r_{n+1}/r_n\to\theta$ for some $\theta\in\R_{>1}\cup\{\infty\}$. This is sadly false; in fact, the class of sparse predicates is not very rigid at all. As an example, fixing $d\in\N_{\geq 2}$, recall that $d^\N=\{d^n:n\in\N\}$ is sparse. However, $T:=\{d^n+1: n\in\N\}$ is not sparse, even though $(d^{n+1}+1)/(d^n+1)\to d$. Indeed, the operator $A$ given by $-\sigma^1+d\sigma^0$ is the constant function with image $\{d-1\}$, so $A>_T 0$, but for all $\Delta\in\N$, $A\sigma^\Delta z<z$ for cofinitely many $z\in T$.
    
    Thus, the condition $r_{n+1}/r_n\to\theta >1$ emphatically fails to be \textit{sufficient} for the sparsity of $R$. However, it transpires to be \textit{necessary}, and more can be said --- see Section \ref{morecanbesaid}.
\end{remark}

For $\textbf{A}=(A_1, ..., A_n)$ an $n$-tuple of operators and $z=(z_1, ..., z_n)\in R^n$, we will write $\textbf{A}\cdot z$ for the dot product of $\textbf{A}$ and $z$: that is, $\textbf{A}\cdot z=A_1z_1+\cdots+A_nz_n$.

We now state and prove some basic results about sparse predicates. Among others, our main goal is to show that if $\textbf{A}$ is an $n$-tuple of non-zero operators, then $z\mapsto \textbf{A}\cdot z$ defines an injective function on a natural subset of $R^n$ (Lemma \ref{ordering}).

For the rest of this subsection, fix a sparse predicate $R\subseteq \N$.
\patchcmd{\thmhead}{(#3)}{#3}{}{}
\begin{lemma}[{\cite[Lemma~2]{semenov}}]\label{beatsme}
    Let $A,B$ be operators with $A\neq_R 0$. Then, for $\Delta\in\N$ sufficiently large, $|A\sigma^\Delta z|>Bz$ for all $z\in R$.
\end{lemma}
\begin{defn}
    Let $\tilde{R}\subseteq R$. For $n,\Delta\in\N$, write
\[\tilde{R}^n_\Delta:=\{(z_1, ..., z_n)\in \tilde{R}^n: z_i\geq \sigma^{\Delta}z_{i+1}\text{ for all }1\leq i\leq n\},\]
where $z_{n+1}:=\min \tilde{R}$.
\end{defn}
\begin{lemma}\label{domination}
    Let $n\in\N^+$, $\textbf{A}$ be an $n$-tuple of operators such that $A_1\neq _R 0$, and $\varepsilon>0$. Then, for all $\Delta\in \N$ sufficiently large and $z\in R^n_\Delta$, we have
    \[(1-\varepsilon)\abs{A_1z_1}<\abs{\textbf{A}\cdot z}<(1+\varepsilon)\abs{A_1z_1}.\]
\end{lemma}
\begin{proof}
    By Lemma \ref{beatsme}, there is $\Lambda\in\N$ such that for all $z_2, ..., z_n\in R$,
    \[\abs{(A_2, ..., A_n)\cdot (z_2, ..., z_n)}\leq \abs{A_2z_2}+\cdots+\abs{A_nz_n}<\sigma^\Lambda z_2 + \cdots + \sigma^\Lambda z_n,\]
    whence for all $\Delta\in \N$ and $z\in R^n_\Delta$, $\abs{(A_2, ..., A_n)\cdot (z_2, ..., z_n)}<n\sigma^{-\Delta+\Lambda}(z_1)$. Thus, by Lemma \ref{beatsme}, if $\Delta\in\N$ is sufficiently large then $\abs{(A_2, ..., A_n)\cdot (z_2, ..., z_n)}<\varepsilon\abs{A_1z_1}$ for all $z\in R^n_\Delta$.
\end{proof}
\begin{lemma}\label{monotone}
    Let $A$ be an operator. If $A>_R 0$ (respectively $A<_R 0$) then there is $r\in \Q_{>1}$ such that $A\sigma z> rAz$ (respectively $A\sigma z< rAz$) for cofinitely many $z\in R$. In particular, the function $R\to R, z\mapsto Az$ is eventually strictly increasing (respectively decreasing).
\end{lemma}
\begin{proof}
    We prove the lemma assuming $A>_R 0$; the case where $A<_R 0$ is similar. By Lemma \ref{beatsme}, there is $\Delta\in\N$ such that $A\sigma^\Delta z>2Az$ for all $z\in R$. Fix $r\in \Q_{>1}$ such that $r^\Delta<2$; write $r=p/q$ for $p,q\in\N^+$. Let $B$ be the operator defined by $Bz = qA\sigma z-pAz$. If $B\leq_R 0$ then $A\sigma z\leq rAz$ for cofinitely many $z\in R$, whence $A\sigma^\Delta z\leq r^\Delta Az<2Az$ for cofinitely many $z\in R$, a contradiction. By (S1), we must thus have that $B>_R 0$, whence $A\sigma z>rAz$ for cofinitely many $z\in R$.
\end{proof}

Here and henceforth, given an $n$-tuple $\nu=(\nu_1,...,\nu_n)$ and $1\leq i\leq n$, we let $\nu_{>i}$ denote $(\nu_{i+1},...,\nu_n)$, $\nu_{\geq i}$ denote $(\nu_i,...,\nu_n)$, and so on.
\begin{lemma}\label{ordering}
    Let $n\in\N^+$, $\textbf{A}$ be an $n$-tuple of operators, and $\Delta\in \N$ be sufficiently large.
    
    Let $z,w\in R^n_\Delta$ be such that $i:=\min\{1\leq e\leq n: z_e\neq w_e, A_e\neq 0\}$ is well-defined, and suppose $z_i>w_i$. Then $\textbf{A}\cdot z>\textbf{A}\cdot w$ if $A_i>0$, and $\textbf{A}\cdot z<\textbf{A}\cdot w$ if $A_i<0$.
    
    In particular, if $\textbf{A}$ is a tuple of non-zero operators, then $z\mapsto \textbf{A}\cdot z$ defines an injective function on $R^n_\Delta$.
\end{lemma}
\begin{proof}
    We prove this assuming $A_i>0$; the case where $A_i<0$ is similar. By Lemma \ref{monotone}, there is $r\in\Q_{>1}$ such that $A_i\sigma x>rA_ix$ for sufficiently large $x\in R$, say for $x\geq \sigma^\Delta(\min R)$, taking $\Delta\in\N$ to be sufficiently large. Let $k\in\N^+$ be such that $r>1+1/k$. By Lemma \ref{domination}, taking $\Delta\in\N$ to be sufficiently large, we have
    \[\textbf{A}_{\geq i}\cdot z_{\geq i}>\left(1-\frac{1}{4k}\right)A_iz_i>\left(1-\frac{1}{4k}\right)\left(1+\frac{1}{k}\right)A_iw_i\geq \left(1+\frac{1}{2k}\right)A_iw_i>\textbf{A}_{\geq i}\cdot w_{\geq i},\]
    where the second inequality is due to the fact that $A_i\sigma x>rA_ix$ for $x\geq \sigma^\Delta(\min R)$, and $w_i\geq \sigma^{n\Delta}(\min R)$ since $w\in R^n_\Delta$. But now
    \[\textbf{A}\cdot z=\textbf{A}_{<i}\cdot z_{<i}+\textbf{A}_{\geq i}\cdot z_{\geq i}=\textbf{A}_{<i}\cdot w_{<i}+\textbf{A}_{\geq i}\cdot z_{\geq i}>\textbf{A}_{<i}\cdot w_{<i}+\textbf{A}_{\geq i}\cdot w_{\geq i}=\textbf{A}\cdot w.\qedhere\]
\end{proof}
\begin{remark}\label{pantelisfavouriteremark}
    In this paper, we frequently consider tuples $z\in R^n_\Delta$ for some sufficiently large $\Delta\in\N$ rather than $z\in R^n$. The reason for this is that, as shown in the preceding lemmas, $R^n_\Delta$ is much better-behaved than $R^n$. We illustrate this by considering Lemma \ref{ordering} for the sparse predicate $R=\twoN$.
    
    As shown in Example \ref{dNoperator}, in this context an operator is simply multiplication by a constant, so let us consider the 3-tuple of operators $\textbf{A}=(1,2,4)$, where $4$ denotes multiplication by $4$, and so on. Lemma \ref{ordering} says, if $\Delta\in\N$ is sufficiently large, then the function $z\mapsto \textbf{A}\cdot z$ is injective on $(\twoN)^3_\Delta$. In other words, if $x=z_1+2z_2+4z_3$ for some $z\in (\twoN)^3_\Delta$, then we can read off $z_1$, $z_2$, and $z_3$ uniquely from $x$. The following is an example to illustrate the \textit{necessity} of $\Delta$ being sufficiently large:
    \[96=1(32)+2(16)+4(8)=1(64)+2(8)+4(4).\]
    On the other hand, the \textit{sufficiency} of $\Delta$ being sufficiently large ($\Delta\geq 2$) is clear from the uniqueness of binary expansions, and indeed this is a special case of Lemma \ref{ordering}.
\end{remark}
\subsection{The $P_\Delta(\cdot;\textbf{A},\tilde{R})$ and $Q_\Delta(\cdot;\textbf{A},\tilde{R})$ functions}
In this subsection, we introduce two functions that are crucial for the rest of the paper. Throughout this subsection, fix a sparse predicate $R\subseteq\N$, enumerated by the increasing sequence $(r_n: n\in\N)$.
\begin{defn}
    Let $d\in\N^+$, and let $\tilde{R}\subseteq R$ be definable in $(\Z,<,+,R)$. Write $\tilde{R}\subseteq^d R$ if there is $N\subseteq\N$ such that
    \[\tilde{R}:=\{r_{N+dt}: t\in\N\}.\]
\end{defn}
This definition is motivated by the following lemma.
\begin{lemma}\label{tildemotivate}
    Let $m,d\in\N^+$, and suppose $R$ is eventually periodic mod $m$ with minimum period $d$. Then, for all $N\in\N$, the set $\tilde{R}:=\{r_{N+dt}: t\in\N\}\subseteq R$ is definable in $(\Z,<,+,R)$, and thus $\tilde{R}\subseteq^d R$.
\end{lemma}
\begin{proof}
    Up to excluding finitely many elements from $\tilde{R}$ (which does not affect the definability of $\tilde{R}$), we may assume that $(r_n: n\geq N)$ is periodic mod $m$. Then, for $z\in R$,
    \[z\in\tilde{R}\Leftrightarrow z\geq r_N \wedge \bigwedge_{p=0}^{d-1} \sigma^p z\equiv \sigma^p r_N \mod m,\]
    and so $\tilde{R}$ is definable in $(\Z,<,+,R)$.
\end{proof}
\begin{defn}\label{defnmM}
Let $n\in\N^+$, $\textbf{A}$ be an $n$-tuple of non-zero operators, and $\Delta\in\N$ be sufficiently large such that the function $z\mapsto \textbf{A}\cdot z$ is injective on $R^n_\Delta$. For $S\subseteq R^n_\Delta$, write $\textbf{A}\cdot S:=\{\textbf{A}\cdot z: z\in S\}$. For $\emptyset\neq S\subseteq R^n_\Delta$ such that $\textbf{A}\cdot S$ is bounded below, let
    \[\min_{\textbf{A}}S:= \text{ the unique }z\in S\text{ such that }\textbf{A}\cdot z=\min \textbf{A}\cdot S.\]
    Similarly, for $\emptyset\neq S\subseteq R^n_\Delta$ such that $\textbf{A}\cdot S$ is bounded above, let
    \[\max_{\textbf{A}}S:= \text{ the unique }z\in S\text{ such that }\textbf{A}\cdot z=\max \textbf{A}\cdot S.\]
\end{defn}
\begin{defn}\label{PDeltadefn}
    Let $d,n\in\N^+$, $\tilde{R}\subseteq^d R$, $\textbf{A}$ be an $n$-tuple of non-zero operators, and $\Delta\in d\N$ be sufficiently large such that the function $z\mapsto \textbf{A}\cdot z$ is injective on $R^n_\Delta$. For $x\in\Z$, let
    \begin{align*}
        P_\Delta(x;\textbf{A},\tilde{R})&:=\begin{cases}
        \max_{\textbf{A}} \{z\in \tilde{R}^n_\Delta: \textbf{A}\cdot z<x\}&\text{if }x>\inf \textbf{A}\cdot \tilde{R}^n_\Delta,\\
        \min_{\textbf{A}} \tilde{R}^n_\Delta &\text{otherwise},
    \end{cases}\\
        Q_\Delta(x;\textbf{A},\tilde{R})&:=\begin{cases}
        \min_{\textbf{A}} \{z\in \tilde{R}^n_\Delta: \textbf{A}\cdot z\geq x\}&\text{if }x\leq \sup \textbf{A}\cdot \tilde{R}^n_\Delta,\\
        \max_{\textbf{A}} \tilde{R}^n_\Delta&\text{otherwise}.
    \end{cases}
    \end{align*}
    For $1\leq i\leq n$, write $P^i_\Delta(x;\textbf{A},\tilde{R})$ for $P_\Delta(x;\textbf{A},\tilde{R})_i$ and $Q^i_\Delta(x;\textbf{A},\tilde{R})$ for $Q_\Delta(x;\textbf{A},\tilde{R})_i$. The parameter $\tilde{R}$ is dropped where obvious from context.
\end{defn}
\begin{remark}\label{Premark}
    \begin{enumerate}[(i)]
        \item In other words, if $x>\inf \textbf{A}\cdot \tilde{R}^n_\Delta$, then $P_\Delta(x;\textbf{A},\tilde{R})$ is the element $z\in \tilde{R}^n_\Delta$ maximising $\textbf{A}\cdot z$ subject to $\textbf{A}\cdot z<x$. Similarly, if $x\leq \sup \textbf{A}\cdot \tilde{R}^n_\Delta$, then $Q_\Delta(x;\textbf{A},\tilde{R})$ is the element $z\in \tilde{R}^n_\Delta$ minimising $\textbf{A}\cdot z$ subject to $\textbf{A}\cdot z\geq x$.
        \item If $x\leq \inf \textbf{A}\cdot \tilde{R}^n_\Delta$, then $A_1>_R 0$ (as otherwise $\inf \textbf{A}\cdot \tilde{R}^n_\Delta=-\infty$). In this case, by Lemma \ref{ordering}, $P_\Delta(x;\textbf{A},\tilde{R})=\min_{\textbf{A}} \tilde{R}^n_\Delta$ is the lexicographically minimal element of $\tilde{R}^n_\Delta$, namely,
        \[(\sigma^{(n-i)\Delta}(\min \tilde{R}): 0\leq i<n).\]
        Similarly, if $x>\sup \textbf{A}\cdot \tilde{R}^n_\Delta$, then $A_1<_R 0$ and
        \[Q_\Delta(x;\textbf{A},\tilde{R})=\max_{\textbf{A}} \tilde{R}^n_\Delta=(\sigma^{(n-i)\Delta}(\min \tilde{R}): 0\leq i<n).\]
    \end{enumerate}
\end{remark}
\begin{example}
    As in Remark \ref{pantelisfavouriteremark}, consider the example $R=\twoN$ and $\textbf{A}=(1,2,4)$. Let $\Delta=2$; it is easy to verify that $z\mapsto \textbf{A}\cdot z$ is injective on $R^3_2$. The first four elements of $\textbf{A}\cdot R^3_2$ are
    \[1(16)+2(4)+4(1)=28, 1(32)+2(4)+4(1)=44, 1(32)+2(8)+4(1)=52, 1(32)+2(8)+4(2)=56.\]
    Since $44<47\leq 52$, we have $P_2(47;\textbf{A},R)=(32,4,1)$ and $Q_2(47;\textbf{A},R)=(32,8,1)$. Furthermore, for all $x\leq 28=\inf \textbf{A}\cdot R^3_2$, we have $P_2(x;\textbf{A},R)=(16,4,1)$.
\end{example}
The following lemma establishes basic properties of $P_\Delta(\cdot;\textbf{A},\tilde{R})$ and $Q_\Delta(\cdot;\textbf{A},\tilde{R})$. The proofs are rather straightforward but we include them to provide more intuition on these functions.
\begin{lemma}\label{Plemma}
    Let $\tilde{R}\subseteq^d R$ for some $d\in\N^+$. Let $n\in\N^+$, $\textbf{A}$ be an $n$-tuple of non-zero operators, and $\Delta\in d\N$ be sufficiently large. Then the following hold.
    \begin{enumerate}[(i)]
        \item For all $x\in\Z$, $x>\textbf{A}\cdot P_\Delta(x;\textbf{A},\tilde{R})$ if and only if $x>\inf \textbf{A}\cdot \tilde{R}^n_\Delta$, and $x\leq\textbf{A}\cdot Q_\Delta(x;\textbf{A},\tilde{R})$ if and only if $x\leq\sup \textbf{A}\cdot \tilde{R}^n_\Delta$.
        \item For all $x\in\Z$, $Q^1_\Delta(x;\textbf{A},\tilde{R})=\sigma^{\varepsilon d} P^1_\Delta(x;\textbf{A},\tilde{R})$ for some $\varepsilon\in\{-1,0,1\}$.
    \end{enumerate}
\end{lemma}
\begin{proof}
    We first prove (i). If $x\leq \inf \textbf{A}\cdot \tilde{R}^n_\Delta$, then $x\leq \textbf{A}\cdot P_\Delta(x;\textbf{A})$ since $P_\Delta(x;\textbf{A})\in \tilde{R}^n_\Delta$. If $x> \inf \textbf{A}\cdot \tilde{R}^n_\Delta$, then by definition $P_\Delta(x;\textbf{A})\in\{z\in \tilde{R}^n_\Delta: \textbf{A}\cdot z<x\}$, so $x>\textbf{A}\cdot P_\Delta(x;\textbf{A})$. The corresponding statement for $Q_\Delta(\cdot;\textbf{A})$ can be proven similarly.
    
    We now prove (ii). If $x\leq \inf \textbf{A}\cdot \tilde{R}^n_\Delta$, then
    \[Q_\Delta(x;\textbf{A})=\min_{\textbf{A}}\{z\in \tilde{R}^n_\Delta: \textbf{A}\cdot z\geq x\}=\min_{\textbf{A}} \tilde{R}^n_\Delta = P_\Delta(x;\textbf{A}).\]
    Similarly, if $x> \sup \textbf{A}\cdot \tilde{R}^n_\Delta$, then $Q_\Delta(x;\textbf{A})=P_\Delta(x;\textbf{A})$, so consider the case where $\inf \textbf{A}\cdot \tilde{R}^n_\Delta<x\leq \sup \textbf{A}\cdot \tilde{R}^n_\Delta$. Then by definition and part (i) we have that $\textbf{A}\cdot P_\Delta(x;\textbf{A})<x\leq \textbf{A}\cdot Q_\Delta(x;\textbf{A})$, and there is no $z\in \tilde{R}^n_\Delta$ such that $\textbf{A}\cdot P_\Delta(x;\textbf{A})<\textbf{A}\cdot z< \textbf{A}\cdot Q_\Delta(x;\textbf{A})$. The statement now follows from Lemma \ref{ordering}.
\end{proof}
\subsection{Sparsity as regularity}\label{morecanbesaid}
We conclude this section by proving that the notion of a sparse predicate coincides with that of a \textit{regular} predicate, defined by Lambotte and Point in \cite{regularnip} and recalled below.
\begin{defn}[\cite{regularnip}]
    Let $R\subseteq \N$ be enumerated by the increasing sequence $(r_n: n\in\N)$. Say that $R$ is \textit{regular} if $r_{n+1}/r_n\to \theta\in \R_{>1}\cup\{\infty\}$ and, if $\theta$ is algebraic over $\Q$ with minimal polynomial $f(x)$, then the operator $f(\sigma)=_R 0$, that is, if $f(x)=\sum_{i=0}^k a_i x^i$ then for all $n\in\N$ we have
    \[\sum_{i=0}^k a_ir_{n+i}=0.\]
\end{defn}
Lambotte and Point prove that regular predicates are sparse \cite[Lemma~2.26]{regularnip}. It turns out that these notions coincide.
\begin{theorem}\label{sparseisregular}
    Let $R\subseteq\N$. Then $R$ is sparse if and only if $R$ is regular.
\end{theorem}
\begin{proof}
    It suffices to prove the forward direction. Let $R$ be a sparse predicate, enumerated by the increasing sequence $(r_n: n\in\N)$. If $\liminf_{n\to\infty} r_{n+1}/r_n\neq \limsup_{n\to\infty} r_{n+1}/r_n$, then there is $p\in\Q_{>1}$ such that $\{n\in\N: r_{n+1}/r_n>p\}$ and $\{n\in\N: r_{n+1}/r_n<p\}$ are both infinite. But now, writing $p=a/b$ for $a,b\in\N^+$, the operator $A$ given by $z\mapsto b\sigma z-az$ satisfies that $Az>0$ for infinitely many $z\in R$ and $Az<0$ for infinitely many $z\in R$, a contradiction to (S1).
    
    Thus, $r_{n+1}/r_n\to \theta$ for some $\theta\in \R_{\geq 1}\cup\{\infty\}$. By Lemma \ref{monotone} applied to the identity operator, there is $q\in\Q_{>1}$ such that $r_{n+1}/r_n>q$ for all sufficiently large $n$, so $\theta\neq 1$. Suppose $\theta$ is algebraic over $\Q$ with minimum polynomial $f(x)=\sum_{i=0}^k a_i x^i$. Towards a contradiction, suppose $f(\sigma)\neq_R 0$. Let $g:=f$ if $f(\sigma)>_R 0$, and $g:=-f$ if $f(\sigma)<_R 0$. Then, $g(\sigma)>_R 0$, so by (S2), there is $\Delta\in\N$ such that $g(\sigma)r_{n+\Delta}>r_n$ for all $n\in\N$. But
    \[\frac{g(\sigma)r_n}{r_n}= \pm\sum_{i=0}^k a_i \left(\frac{r_{n+i}}{r_n}\right)\to \pm\sum_{i=0}^k a_i \theta^i= 0,\]
    and $r_{n+\Delta}/r_n\to \theta^\Delta$, so
    \[\frac{g(\sigma)r_{n+\Delta}}{r_n}=\frac{g(\sigma)r_{n+\Delta}}{r_{n+\Delta}}\frac{r_{n+\Delta}}{r_n}\to 0,\]
    contradicting the fact that $g(\sigma)r_{n+\Delta}>r_n$ for all $n\in\N$.
\end{proof}
We find that the notion of regularity gives better intuition for what a sparse/regular predicate looks like, so this is a pleasant result.
\section{Reduction to representative formulas}\label{sectionreduction}
The goal of this section is to find formulas for which constructing strong honest definitions is sufficient for the distality of the structure; this is achieved in Theorem \ref{sufficiency}.\\

\textbf{For the rest of the paper, fix a congruence-periodic sparse predicate $R\subseteq\N$.} (Recall that $R$ is \textit{congruence-periodic} if, for all $m\in\N^+$, $R$ is eventually periodic mod $m$.)\\

Let $\mathcal{L}^0:=(<,+)$ and $\mathcal{L}:=(<,+,R)$.
\begin{defn}
    Let $\phi(x;y)$ be an $\mathcal{L}$-formula with $\abs{x}=1$. Say that $\phi(x;y)$ is a \textit{basic} formula if it is a Boolean combination of formulas not involving $x$ and descendants of $\mathcal{L}^0$-formulas.
\end{defn}
Note that basic formulas have strong honest definitions by Example \ref{LOexample}, Lemma \ref{descendants}, and the fact that formulas not involving $x$ have $\top$ as a strong honest definition.

For $n\in\N^+$ and $1\leq i\leq n$, let $\textbf{F}^i$ be the `$i^{\text{th}}$ standard $n$-tuple of operators' (where $n$ is assumed to be obvious from context): for $1\leq j\leq n$,
\[\textbf{F}^i_j=\begin{cases}
    \text{the identity function}&\text{if }j=i,\\
    0&\text{if }j\neq i.
\end{cases}\]
\begin{defn}\label{defnformulas}
    Let $d,n\in\N^+$, $\tilde{R}\subseteq^d R$, and $\phi(x;...)$ be an $\mathcal{L}$-formula with $\abs{x}=1$.
    
    Let $y$ be a tuple of variables. Say that $\phi=\phi(x;y)$ is of the form $(E_n; \tilde{R})$, or just $(E_n)$, if
    \[\phi(x;y)=\exists z\in \tilde{R}^n_0 \bigwedge_{j=1}^k f_j(x,y)>\textbf{A}^{(j)}\cdot z,\]
    where $f_1(x,y), ..., f_k(x,y)$ are $\Z$-affine functions, and $\textbf{A}^{(1)}, ..., \textbf{A}^{(k)}$ are $n$-tuples of operators. 

    Let $\Delta\in d\N$, $y_1, y_2$ be singleton variables, and $\textbf{A},\textbf{B}$ be $n$-tuples of operators.
    
    Say that $\phi=\phi(x;y_1,y_2)$ is of the form $(F_n; \textbf{A}, \textbf{B}, \tilde{R}, \Delta)$ if $\textbf{A}$ is a tuple of non-zero operators and
    \[\phi(x;y_1,y_2)=tx-y_2<\textbf{B}\cdot P_\Delta(x-y_1;\textbf{A},\tilde{R}),\]
    where $t\in\{0,1\}$ with $t=1$ unless $\textbf{B}=\textbf{F}^i$ for some $1\leq i\leq n$.

    Let $u,v$ be $n$-tuples of variables, and let $T_{\tilde{R}}(u,v)$ be the formula $u_1, v_1, ..., u_n, v_n\in \tilde{R}$. Say that $\phi=\phi(x;y_1,y_2,u,v)$ is of the form $(G_n; \textbf{A}, \textbf{B}, \tilde{R}, \Delta)$ if either
    \[\phi(x;y_1,y_2,u,v)=T_{\tilde{R}}(u,v)\wedge \exists z\in \tilde{R}^n_\Delta\bigg(y_1+\textbf{A}\cdot z < x<y_2+\textbf{B}\cdot z\wedge \bigwedge_{i=1}^n u_i\leq z_i\leq v_i\bigg),\]
    or $\phi$ is obtained from the formula above by deleting some of the $u_i$ (equivalently, setting $u_i=-\infty$) and/or deleting some of the $v_i$ (equivalently, setting $v_i=+\infty$).
\end{defn}
It will be convenient to extend the definition of $(E_n)$ formulas to $n=0$; that is, $\phi(x;y)$ with $\abs{x}=1$ is of the form $(E_0)$ if
    \[\phi(x;y)=\bigwedge_{j=1}^k f_j(x,y)>0,\]
    where $f_1(x,y), ..., f_k(x,y)$ are $\Z$-affine functions. Note that such formulas are basic.

Our goal is to prove the following theorem.
    
\begin{theorem}\label{sufficiency}
    The following criterion is sufficient for the distality of $(\Z,<,+,R)$.

    Let $d,n\in\N^+$, $\tilde{R}\subseteq^d R$, $\textbf{A}$ be an $n$-tuple of non-zero operators, and $\textbf{B}$ be an $n$-tuple of operators. Then, for all sufficiently large $\Delta\in d\N$, every $(F_n;\textbf{A},\textbf{B},\tilde{R},\Delta)$ formula has a strong honest definition.
\end{theorem}
We prove this in three steps. We first show that every $\mathcal{L}$-formula $\phi(x;y)$ with $\abs{x}=1$ is equivalent to a Boolean combination of basic formulas and descendants of $(E_n)$ formulas (Proposition \ref{agenda}). We then show that every $(E_n)$ formula is equivalent to a Boolean combination of basic formulas and descendants of $(E_{n-1})$ or $(G_n; ...)$ formulas (Proposition \ref{goingup}). Finally, we show that every $(G_n; ...)$ formula is equivalent to a Boolean combination of basic formulas, $(E_{n-1})$ formulas, and descendants of $(F_n; ...)$ formulas (Proposition \ref{EnFnimpliesGn}).

Our first checkpoint is the following proposition.
\begin{prop}\label{agenda}
    Modulo $(\Z,<,+,R)$, every formula $\phi(x;y)$ with $\abs{x}=1$ is equivalent to a Boolean combination of basic formulas and descendants of $(E_n)$ formulas.
\end{prop}
To prove this, we recall the following result of Semenov.
\patchcmd{\thmhead}{(#3)}{#3}{}{}
\begin{theorem}[{\cite[Theorem~3]{semenov}}]\label{existential}
    Modulo $(\Z,<,+,R)$, every formula $\phi(x)$ is equivalent to a disjunction of formulas of the form
    \[\exists z\in R^n \left(\bigwedge_{j=1}^k f_j(x)>\textbf{A}^{(j)}\cdot z \wedge \bigwedge_{p=1}^l g_p(x)\equiv \textbf{B}^{(p)}\cdot z\mod m_p\right),\]
    where $m_p\in\N^+$, $f_j(x), g_p(x)$ are $\Z$-affine functions, and $\textbf{A}^{(j)}, \textbf{B}^{(p)}$ are $n$-tuples of operators.
\end{theorem}
\patchcmd{\thmhead}{#3}{(#3)}{}{}
\begin{remark}\label{gapremark}
    In \cite[Theorem~3]{semenov}, $R\subseteq \N$ is only assumed to be sparse, not necessarily congruence-periodic. However, it would appear that Semenov's proof only goes through if congruence-periodicity is assumed. The statement of \cite[Theorem~3]{semenov} is for a larger class of theories which Semenov calls $\textbf{T}^*(\N,+,R,\mathcal{R})$. His proof uses \cite[Theorem~2]{semenov}, which is a statement for theories which Semenov calls $\textbf{T}^*(\N,>,\mathcal{P},R,\mathcal{R})$, where $\mathcal{P}$ is an eventually periodic set of predicates. Semenov applies this theorem with $\mathcal{P}=\mathcal{E}:=\{x\equiv c\; (\text{mod } m): c,m\in\N\}$. However, the proof of \cite[Theorem~2]{semenov} appears to use the fact that the predicates $(P\cap R: P\in \mathcal{P})$ are also eventually periodic in order to apply \cite[Theorem~1]{semenov}; in the case that $\mathcal{P}=\mathcal{E}$, this translates to the requirement that $R$ is congruence-periodic.

    An alternative viewpoint that casts doubt on the validity of Semenov's result (without assuming congruence-periodicity of $R$) is that this would imply that the formula $\neg \exists z\in R\;(x_1<z<x_2\wedge z\equiv c\mod m)$ is equivalent to an existential formula without any congruence-periodicity assumptions. We are unable to recover a proof of this from \cite{semenov} without assuming that $R$ is congruence-periodic, and our intuition is that this is false.
\end{remark}
\begin{proof}[Proof of Proposition \ref{agenda}]
By Theorem \ref{existential}, every partitioned $\mathcal{L}$-formula $\phi(x;y)$ with $\abs{x}=1$ is equivalent to a disjunction of formulas of the form
\[\exists z\in R^n \left(\bigwedge_{j=1}^k f_j(x,y)>\textbf{A}^{(j)}\cdot z \wedge \bigwedge_{p=1}^l g_p(x,y)\equiv \textbf{B}^{(p)}\cdot z\mod m_p\right),\]
    where $m_p\in\N^+$, $f_j(x,y), g_p(x,y)$ are $\Z$-affine functions, and $\textbf{A}^{(j)}, \textbf{B}^{(p)}$ are $n$-tuples of operators. By the Chinese Remainder Theorem, it suffices to assume that there is $m\in\N^+$ such that $m=m_p$ for all $1\leq p\leq l$.
    
    It suffices to show that every such formula is equivalent to a Boolean combination of basic formulas and descendants of $(E_s)$ formulas for some $s\in\N$. We do so by induction on $n\in\N$. When $n=0$, the formula is a basic formula. Now let $n\geq 1$, and let
    \[\phi(x,y):=\exists z\in R^n \left(\bigwedge_{j=1}^k f_j(x,y)>\textbf{A}^{(j)}\cdot z \wedge \bigwedge_{p=1}^l g_p(x,y)\equiv \textbf{B}^{(p)}\cdot z\mod m\right),\]
    where $m\in\N^+$, $f_j(x,y), g_p(x,y)$ are $\Z$-affine functions, and $\textbf{A}^{(j)}, \textbf{B}^{(p)}$ are $n$-tuples of operators.
    
    Let $(r_n: n\in\N)$ be an increasing enumeration of $R$. Since $R$ is congruence-periodic, there are $d,N\in\N$ such that $(r_n: n\geq N)$ is periodic mod $m$ with minimum period $d$. Observe that $\phi(x;y)$ is equivalent to $\phi_0(x;y)\vee \phi_1(x;y)$, where
\begin{align*}
    \phi_0(x;y)&:=\bigvee_{i=1}^n \bigvee_{\alpha=0}^{N-1} \exists z\in R^n \left(z_i=r_\alpha\wedge \bigwedge_{j=1}^k f_j(x,y)>\textbf{A}^{(j)}\cdot z \wedge \bigwedge_{p=1}^l g_p(x,y)\equiv \textbf{B}^{(p)}\cdot z\mod m\right),\\
    \phi_1(x;y)&:=\exists z\in R^n \left(\bigwedge_{i=1}^n z_i\geq r_N\wedge\bigwedge_{j=1}^k f_j(x,y)>\textbf{A}^{(j)}\cdot z \wedge \bigwedge_{p=1}^l g_p(x,y)\equiv \textbf{B}^{(p)}\cdot z\mod m\right).
\end{align*}
Consider $\phi_0(x;y)$. Replacing $z_i$ with $r_\alpha$ in the $(i,\alpha)^{\text{th}}$ disjunct, $\phi_0(x;y)$ is equivalent to a disjunction of formulas of the form
\[\exists w\in R^{n-1} \left(\bigwedge_{j=1}^k f'_j(x,y)>\textbf{A}'^{(j)}\cdot w \wedge \bigwedge_{p=1}^l g'_p(x,y)\equiv \textbf{B}'^{(p)}\cdot w\mod m\right),\]
where $f'_j(x,y), g'_p(x,y)$ are $\Z$-affine functions and $\textbf{A}'^{(j)}, \textbf{B}'^{(p)}$ are $(n-1)$-tuples of operators. By the induction hypothesis, such formulas are equivalent to a Boolean combination of basic formulas and descendants of $(E_s)$ formulas for some $s\in\N$.

Consider $\phi_1(x;y)$. Let $\tilde{R}:=\{r_{N+dt}: t\in\N\}$. By Lemma \ref{tildemotivate}, $\tilde{R}\subseteq^d R$. For $1\leq p\leq l$ and $0\leq h_1, ..., h_n<d$, let $0\leq b^{(p)}_{h_1, ..., h_n}<m$ be such that
\[\textbf{B}^{(p)}\cdot (r_{N+h_1}, ..., r_{N+h_n})\equiv b^{(p)}_{h_1, ..., h_n}\mod m.\]
Now $\phi_1(x;y)$ is equivalent to
\[\bigvee_{0\leq h_1, ..., h_n< d} \left(\bigwedge_{p=1}^l g_p(x,y)\equiv b^{(p)}_{h_1, ..., h_n}\mod m\wedge \exists z\in \tilde{R}^n\bigwedge_{j=1}^k f_j(x,y)>\textbf{A}^{(j)}\cdot (\sigma^{h_1}z_1, ..., \sigma^{h_n}z_n)\right).\]
But now, for all $0\leq h_1, ..., h_n<d$,
\begin{align*}
   &\exists z\in \tilde{R}^n\bigwedge_{j=1}^k f_j(x,y)>\textbf{A}^{(j)}\cdot (\sigma^{h_1}z_1, ..., \sigma^{h_n}z_n)\\
   &\Leftrightarrow\bigvee_{\tau\in\text{Sym}(n)}\exists z\in \tilde{R}^n \left(\bigwedge_{i=1}^{n-1}z_{\tau(i)}\geq z_{\tau(i+1)}\wedge \bigwedge_{j=1}^k f_j(x,y)>\textbf{A}^{(j)}\cdot (\sigma^{h_1}z_1, ..., \sigma^{h_n}z_n)\right)\\
   &\Leftrightarrow\bigvee_{\tau\in\text{Sym}(n)}\exists z\in \tilde{R}^n_0 \bigwedge_{j=1}^k f_j(x,y)>\textbf{A}^{(j)}\cdot (\sigma^{h_1}z_{\tau^{-1}(1)}, ..., \sigma^{h_n}z_{\tau^{-1}(n)}),
\end{align*}
so $\phi_1(x;y)$ is equivalent to a Boolean combination of basic formulas and $(E_n; \tilde{R})$ formulas.
\end{proof}

Our next checkpoint is the following proposition.
\begin{prop}\label{goingup}
    Let $d,n\in\N^+$, $\tilde{R}\subseteq^d R$, and $\phi(x;y)$ be an $(E_n; \tilde{R})$ formula. Then there is a finite collection $\mathcal{G}_\phi$ of pairs $(\textbf{A},\textbf{B})$, where $\textbf{A}, \textbf{B}$ are $n$-tuples of operators, satisfying the following.
    
    For all $\Delta\in d\N$ sufficiently large, $\phi$ is equivalent to a Boolean combination of basic formulas and descendants of $(E_{n-1};\tilde{R})$ or $(G_n;\textbf{A},\textbf{B},\tilde{R},\Delta)$ formulas for $(\textbf{A},\textbf{B})\in \mathcal{G}_\phi$.
\end{prop}
Towards this checkpoint, we prove the following technical lemma.
\begin{lemma}\label{uv}
    Let $d,n\in\N^+$, $\tilde{R}\subseteq^d R$, $\textbf{A}^{(1)}, ..., \textbf{A}^{(k)}$ be $n$-tuples of operators, and $\Delta\in d\N$ be sufficiently large. Then there are $1\leq i_1, ..., i_r\leq n$, an $\mathcal{L}^0$-formula $\theta$, and $\mathcal{L}$-definable functions $f_1, ..., f_r, u_1, ..., u_n, v_1, ..., v_n$ such that each $u_i$ (respectively $v_i$) either takes values in $\tilde{R}$ or is the constant $-\infty$ (respectively $+\infty$) function, satisfying that for all $y\in\Z^k$ and $z\in \tilde{R}^n_\Delta$,
    \[\bigwedge_{j=1}^k y_j > \textbf{A}^{(j)}\cdot z \Leftrightarrow \theta(y)\wedge \left(\bigg(\bigwedge_{j=1}^k y_j > \textbf{A}^{(j)}\cdot z\wedge\bigvee_{s=1}^r z_{i_s}=f_s(y)\bigg)\vee\bigg(\bigwedge_{i=1}^n u_i(y)\leq z_i\leq v_i(y)\bigg)\right).\]
\end{lemma}
\begin{proof}
    Let $H_0:=\{1\leq j\leq k: A^{(j)}_i=_R 0\text{ for all }1\leq i\leq n\}$, and for $1\leq i\leq n$, let
    \begin{align*}
        H^+_i&:=\{1\leq j\leq k: A^{(j)}_i>_R 0, A^{(j)}_e=_R 0\text{ for all }e<i\},\\
        H^-_i&:=\{1\leq j\leq k: A^{(j)}_i<_R 0, A^{(j)}_e=_R 0\text{ for all }e<i\},
    \end{align*}
    and write $H_i:=H^-_i\cup H^+_i$. Then $H_0, (H^+_i, H^-_i: 1\leq i\leq n)$ is a partition of $[k]=\{1, ..., k\}$.
    
    Let $1\leq i\leq n$. For all $j\in H^+_i$, define the function $f_j:\Z\to \tilde{R}$ by
    \[f_j(y):=\begin{cases}
        \max\{w\in \tilde{R}: \exists z\in \tilde{R}^n_\Delta (\textbf{A}^{(j)}\cdot z< y\wedge z_i=w)\}&\text{if well-defined},\\
        \min \tilde{R}&\text{otherwise}.
    \end{cases}\]
    By Lemma \ref{ordering}, for all $j\in H^+_i$, $y_j\in\Z^k$, and $z\in \tilde{R}^n_\Delta$, if $z_i<f_j(y_j)$ then $y_j>\textbf{A}^{(j)}\cdot z$, and if $z_i>f_j(y_j)$ then $y_j<\textbf{A}^{(j)}\cdot z$; thus,
    \[y_j>\textbf{A}^{(j)}\cdot z\Leftrightarrow (y_j>\textbf{A}^{(j)}\cdot z\wedge z_i=f_j(y_j))\vee z_i\leq \sigma^{-d}f_j(y_j).\]
    Similarly, for all $j\in H^-_i$, defining the function $f_j:\Z\to \tilde{R}$ by
    \[f_j(y):=\min\{w\in \tilde{R}: \exists z\in \tilde{R}^n_\Delta (\textbf{A}^{(j)}\cdot z< y\wedge z_i=w)\},\]
    we have that, for all $y_j\in\Z^k$ and $z\in \tilde{R}^n_\Delta$,
    \[y_j>\textbf{A}^{(j)}\cdot z\Leftrightarrow (y_j>\textbf{A}^{(j)}\cdot z\wedge z_i=f_j(y_j))\vee z_i\geq \sigma^d f_j(y_j).\]
    For all $1\leq i\leq n$, define $u_i(y):=\sup\{\sigma^d f_j(y): j\in H^-_i\}$ and $v_i(y):=\inf\{\sigma^{-d}f_j(y): j\in H^+_i\}$. Now, if $y_j>\textbf{A}^{(j)}\cdot z$ for all $j\in [k]\setminus H_0$, then either $z_i=f_j(y_j)$ for some $1\leq i\leq n$ and $j\in H_i$, or $u_i(y)\leq z_i\leq v_i(y)$ for all $1\leq i\leq n$. Conversely, if $u_i(y)\leq z_i\leq v_i(y)$ for all $1\leq i\leq n$, then $y_j>\textbf{A}^{(j)}\cdot z$ for all $j\in [k]\setminus H_0$. We conclude that, for all $y_j\in\Z^k$ and $z\in \tilde{R}^n_\Delta$,
    \[\bigwedge_{j=1}^k y_j > \textbf{A}^{(j)}\cdot z\Leftrightarrow\bigwedge_{j\in H_0} y_j > 0 \wedge \Bigg(\bigg(\bigwedge_{j=1}^k y_j > \textbf{A}^{(j)}\cdot z \wedge \bigvee_{i=1}^n \bigvee_{j\in H_i} z_i=f_j(y_j)\bigg)\vee\bigwedge_{i=1}^n u_i(y)\leq z_i\leq v_i(y)\Bigg)\]
    as required.
\end{proof}
\begin{proof}[Proof of Proposition \ref{goingup}]
    Let
    \[\phi(x;y)=\exists z\in \tilde{R}^n_0 \bigwedge_{j=1}^k f_j(x,y)>\textbf{A}^{(j)}\cdot z,\]
    where $\abs{x}=1$, $f_1(x,y), ..., f_k(x,y)$ are $\Z$-affine functions, and $\textbf{A}^{(1)}, ..., \textbf{A}^{(k)}$ are $n$-tuples of operators. We claim that $\mathcal{G}_\phi:=\{(\textbf{A}^{(j)}, -\textbf{A}^{(l)}): 1\leq j, l\leq k\}$ witnesses the proposition.
    
    For all $\Delta\in d\N$, $\phi(x;y)$ is equivalent to the disjunction of
    \[\phi'_\Delta(x;y):=\exists z\in \tilde{R}^n_\Delta \bigwedge_{j=1}^k f_j(x,y)>\textbf{A}^{(j)}\cdot z\]
    and
\[\bigvee_{i=1}^n\bigvee_{\alpha=0}^{\Delta-1}\exists z\in \tilde{R}^n_0\bigg(z_i=\sigma^\alpha z_{i+1}\wedge \bigwedge_{j=1}^k f_j(x;y)>\textbf{A}^{(j)}\cdot z\bigg),\]
where $z_{n+1}:=\min\tilde{R}$. Replacing $z_i$ with $\sigma^\alpha z_{i+1}$ in the $(i,\alpha)^{\text{th}}$ disjunct, it is clear that each disjunct is equivalent to
\[\exists w\in \tilde{R}^{n-1}_0\bigwedge_{j=1}^k f_j(x,y)>\textbf{B}^{(j)}\cdot w,\]
for some $(n-1)$-tuples $\textbf{B}^{(1)}, ..., \textbf{B}^{(k)}$ of operators, which is an $(E_{n-1};\tilde{R})$ formula.

Consider $\phi'_\Delta(x;y)$. By multiplying both sides of the inequalities in $\phi'_\Delta(x;y)$, we may assume without loss of generality that there are $K\in\N^+$ and $0\leq p\leq q\leq k$ such that, for $1\leq j\leq k$,
\[\text{the coefficient of }x\text{ in }f_j=\begin{cases}
    K&\text{if }j\leq p,\\
    -K&\text{if }p<j\leq q,\\
    0&\text{if }q<j.
\end{cases}\]
For $1\leq j\leq k$, let $g_j(y):=f_j(0,y)$. Then $\bigwedge_{j=1}^k f_j(x,y)>\textbf{A}^{(j)}\cdot z$ is (equivalent to)
    \[\bigg(\bigwedge_{j=1}^p -g_j(y)+\textbf{A}^{(j)}\cdot z<Kx \bigg)\wedge \bigg(\bigwedge_{j=p+1}^q Kx<g_j(y)-\textbf{A}^{(j)}\cdot z\bigg)\wedge \bigg(\bigwedge_{j=q+1}^k g_j(y)>\textbf{A}^{(j)}\cdot z\bigg).\]
    If $0=p=q$, then $\phi'_\Delta(x;y)$ is a basic formula. If $0=p<q$, then for all $\Delta\in d\N$,
    \[\phi'_\Delta(x;y) \Leftrightarrow Kx<\sup \bigg\lbrace\inf_{p+1\leq j\leq q} g_j(y)-\textbf{A}^{(j)}\cdot z: z\in \tilde{R}^n_\Delta, \bigwedge_{j=q+1}^k g_i(y)>\textbf{A}^{(i)}\cdot z\bigg\rbrace,\]
    which is a basic formula. The case where $0<p=q$ is similar, so let us assume $0<p<q$. Now $\bigwedge_{j=1}^p -g_j(y)+\textbf{A}^{(j)}\cdot z<Kx$ is equivalent to
    \[\bigvee_{j=1}^p \bigg(-g_j(y)+\textbf{A}^{(j)}\cdot z<Kx\wedge\bigwedge_{\substack{i=1\\i\neq j}}^p -g_j(y)+\textbf{A}^{(j)}\cdot z\geq -g_i(y)+\textbf{A}^{(i)}\cdot z\bigg),\]
    and $\bigwedge_{j=p+1}^q Kx<g_j(y)-\textbf{A}^{(j)}\cdot z$ is equivalent to
    \[\bigvee_{j=p+1}^q \bigg(Kx<g_j(y)-\textbf{A}^{(j)}\cdot z\wedge\bigwedge_{\substack{i=p+1\\i\neq j}}^q g_j(y)-\textbf{A}^{(j)}\cdot z\leq g_i(y)-\textbf{A}^{(i)}\cdot z\bigg).\]
    Thus, for all $\Delta\in d\N$, $\phi'_\Delta(x;y)$ is equivalent to
    \[\bigvee_{j=1}^p\bigvee_{l=p+1}^q \exists z\in \tilde{R}^n_\Delta\bigg(-g_j(y)+\textbf{A}^{(j)}\cdot z<Kx < g_l(y)-\textbf{A}^{(l)}\cdot z\wedge h_{jl}(y,z)\bigg),\]
    where $h_{jl}(y,z)$ is
    \[\bigwedge_{\substack{i=1\\i\neq j}}^p g_i(y)-g_j(y)\geq (\textbf{A}^{(i)}-\textbf{A}^{(j)})\cdot z\wedge \bigwedge_{\substack{i=p+1\\i\neq l}}^q g_i(y)-g_l(y)\geq (\textbf{A}^{(i)}-\textbf{A}^{(l)})\cdot z\wedge \bigwedge_{i=q+1}^k g_i(y)>\textbf{A}^{(i)}\cdot z.\]

Apply Lemma \ref{uv} to each $h_{jl}(y,z)$, assuming $\Delta\in d\N$ is sufficiently large. For all $1\leq j\leq p< l\leq q$, there are $1\leq i^{jl}_1, ..., i^{jl}_{r(j,l)}\leq n$, an $\mathcal{L}^0$-formula $\theta_{jl}$, and $\mathcal{L}$-definable functions $f^{jl}_1, ..., f^{jl}_{r(j,l)},u^{jl}_1, ..., u^{jl}_n, v^{jl}_1, ..., v^{jl}_n$ such that each $u^{jl}_i$ (respectively $v^{jl}_i$) either takes values in $\tilde{R}$ or is the constant $-\infty$ (respectively $+\infty$) function, satisfying that for all $y\in\Z^k$ and $z\in \tilde{R}^n_\Delta$,
    \[h_{jl}(y,z) \Leftrightarrow \theta_{jl}(y)\wedge\left( \bigg(h_{jl}(y,z)\wedge\bigvee_{s=1}^{r(j,l)} z_{i^{jl}_s}=f^{jl}_s(y)\bigg)\vee\bigg(\bigwedge_{i=1}^n u^{jl}_i(y)\leq z_i\leq v^{jl}_i(y)\bigg)\right).\]
    
Then, $\phi'_\Delta(x;y)$ is equivalent to the disjunction of
\[\bigvee_{j=1}^p\bigvee_{l=p+1}^q\bigvee_{s=1}^{r(j,l)} \bigg(\theta_{jl}(y)\wedge \exists z\in \tilde{R}^n_\Delta\Big(-g_j(y)+\textbf{A}^{(j)}\cdot z<Kx < g_l(y)-\textbf{A}^{(l)}\cdot z\wedge h_{jl}(y,z)\wedge z_{i^{jl}_s}=f^{jl}_s(y)\Big)\bigg),\]
which is equivalent to a Boolean combination of basic formulas and descendants of $(E_{n-1};\tilde{R})$ formulas (since $z_{i^{jl}_s}=f^{jl}_s(y)$ in the $(j,l,s)^{\text{th}}$ disjunct), and
\[\bigvee_{j=1}^p\bigvee_{l=p+1}^q \Bigg(\theta_{jl}(y)\wedge\exists z\in \tilde{R}^n_\Delta\bigg(-g_j(y)+\textbf{A}^{(j)}\cdot z<Kx < g_l(y)-\textbf{A}^{(l)}\cdot z\wedge \bigwedge_{i=1}^n u^{jl}_i(y)\leq z_i\leq v^{jl}_i(y)\bigg)\Bigg),\]
which is a Boolean combination of basic formulas and descendants of $(G_n; \textbf{A}, \textbf{B}, \tilde{R}, \Delta)$ formulas for $(\textbf{A}, \textbf{B})\in \mathcal{G}_\phi$.
\end{proof}

Our final checkpoint is the following proposition.
\begin{prop}\label{EnFnimpliesGn}
    Let $n\in\N^+$ and $\textbf{A}, \textbf{B}$ be $n$-tuples of operators. Then there is a finite collection $\mathcal{F}_{\textbf{A}, \textbf{B}}$ of tuples $(\textbf{I}, \textbf{J})$, where $\textbf{I}$ is an $n$-tuple of non-zero operators and $\textbf{J}$ is an $n$-tuple of operators, satisfying the following.
    
    Let $\tilde{R}\subseteq^d R$ for some $d\in\N^+$. If $\Delta\in d\N$ is sufficiently large, then every $(G_n; \textbf{A}, \textbf{B}, \tilde{R}, \Delta)$ formula is equivalent to a Boolean combination of basic formulas, $(E_{n-1};\tilde{R})$ formulas, and descendants of $(F_n;\textbf{I},\textbf{J},\tilde{R},\Delta)$ formulas for $(\textbf{I},\textbf{J})\in \mathcal{F}_{\textbf{A}, \textbf{B}}$.
\end{prop}
Before proving this, we record a corollary.
\begin{cor}\label{EntoFn}
    Let $d,n\in\N^+$, $\tilde{R}\subseteq^d R$, and $\phi(x;y)$ be an $(E_n; \tilde{R})$ formula. Then there is a finite collection $\mathcal{F}_\phi$ of tuples $(\textbf{I}, \textbf{J})$, where $\textbf{I}$ is an $n$-tuple of non-zero operators and $\textbf{J}$ is an $n$-tuple of operators, satisfying the following.
    
    If $\Delta\in d\N$ is sufficiently large, then $\phi(x;y)$ is equivalent to a Boolean combination of basic formulas and descendants of $(E_{n-1};\tilde{R})$ formulas or $(F_n;\textbf{I},\textbf{J},\tilde{R},\Delta)$ formulas for $(\textbf{I},\textbf{J})\in \mathcal{F}_\phi$.
\end{cor}
\begin{proof}
    For $\mathcal{G}_\phi$ from Proposition \ref{goingup}, let $\mathcal{F}_\phi:=\bigcup_{(\textbf{A},\textbf{B})\in\mathcal{G}_\phi}\mathcal{F}_{\textbf{A},\textbf{B}}$ for $\mathcal{F}_{\textbf{A},\textbf{B}}$ from Proposition \ref{EnFnimpliesGn}.
\end{proof}
Towards proving Proposition \ref{EnFnimpliesGn}, we prove the following lemma.
\begin{lemma}\label{canonicalwitness}
    Let $d,n\in\N^+$, $\tilde{R}\subseteq^d R$, $\textbf{A}, \textbf{B}$ be $n$-tuples of operators, and $\Delta\in d\N$ be sufficiently large. Then for all $x,y_1,y_2\in\Z$, $u_i\in \tilde{R}\cup\{-\infty\}$, and $v_i\in \tilde{R}\cup\{+\infty\}$, if
    \begin{equation}\label{canonicalwitnessfmla}\tag{$\dag$}
        \exists z\in \tilde{R}^n_\Delta \bigg(y_1+\textbf{A}\cdot z<x<y_2+\textbf{B}\cdot z\wedge \bigwedge_{i=1}^n u_i\leq z_i\leq v_i\bigg),
    \end{equation}
    then either $v_1=+\infty\wedge (A_1=_R 0<_R B_1 \vee A_1<_R 0=_R B_1)$ or there is a witness $z\in \tilde{R}^n_\Delta$ satisfying one of the following:
    \begin{enumerate}[(i)]
        \item $z_i=\sigma^{\Delta}z_{i+1}$ for some $1\leq i\leq n$, where $z_{n+1}:=\min\tilde{R}$;
        \item $z_i\in\{u_i,v_i\}$ for some $1\leq i\leq n$;
        \item $A_i, B_i\neq_R 0$ for all $1\leq i\leq n$, and $z=P_\Delta(x-y_1;\textbf{A})$ or $z=P_\Delta(y_2-x;-\textbf{B})$.
    \end{enumerate}
\end{lemma}
This lemma has a rather intuitive interpretation: if (\ref{canonicalwitnessfmla}) holds then, barring some edge cases, $z$ can be chosen to satisfy (iii), that is, to maximise $y_1+\textbf{A}\cdot z$ subject to $y_1+\textbf{A}\cdot z<x$ --- namely, $z=P_\Delta(x-y_1;\textbf{A})$ --- or minimise $y_2+\textbf{B}\cdot z$ subject to $x<y_2+\textbf{B}\cdot z$ --- namely, $z=P_\Delta(y_2-x;-\textbf{B})$.
\begin{proof}[Proof of Lemma \ref{canonicalwitness}]
Suppose $v_1\neq +\infty\vee \neg(A_1= 0< B_1 \vee A_1< 0= B_1)$. We first show that if $A_i=0$ or $B_i=0$ for some $1\leq i\leq n$, then there is a witness $z\in \tilde{R}^n_\Delta$ satisfying (i) or (ii).
    
    Suppose $A_i=0$ for some $1\leq i\leq n$; fix the minimal such $i$. Suppose there is no witness to (\ref{canonicalwitnessfmla}) satisfying (i) or (ii). Pick a witness $z\in \tilde{R}^n_\Delta$ that minimises
    \[\begin{cases}
        \min\{z_{i-1},v_i\}/z_i&\text{if }B_i> 0\\
        z_i/\max\{z_{i+1},u_i\}&\text{if }B_i\leq 0
    \end{cases},\]
    where $z_0:=+\infty$ and $z_{n+1}:=\min\tilde{R}$. Let $w$ be the $n$-tuple obtained from $z$ by replacing $z_i$ with $\sigma^d z_i$ (respectively $\sigma^{-d} z_i$) if $B_i> 0$ (respectively $B_i\leq 0$). Since $z$ does not satisfy (i) or (ii), we have that $w\in \tilde{R}^n_\Delta$ and $u_i\leq w_i\leq v_i$. But $\textbf{B}\cdot z\leq \textbf{B}\cdot w$ by Lemma \ref{monotone}, so
    \[y_1+\textbf{A}\cdot w=y_1+\textbf{A}\cdot z<x<y_2+\textbf{B}\cdot z\leq y_2+\textbf{B}\cdot w,\]
    whence $w$ is a witness to $(\dag)$, contradicting our choice of $z$.
    
    The case where $B_i= 0$ for some $1\leq i\leq n$ is similar, so henceforth suppose $A_i, B_i\neq 0$ for all $1\leq i\leq n$, and suppose there is no witness to $(\dag)$ satisfying (i), (ii), or (iii). By Lemma \ref{ordering}, we may assume that the function $z\mapsto \textbf{A}\cdot z$ is injective on $\tilde{R}^n_\Delta$. Now any witness $z\in \tilde{R}^n_\Delta$ to $(\dag)$ satisfies $\textbf{A}\cdot z<x-y_1$ and so $\textbf{A}\cdot z\leq  \textbf{A}\cdot P_\Delta(x-y_1;\textbf{A})$, and the inequality is strict since $z$ does not satisfy (iii). Fix a witness $z\in \tilde{R}^n_\Delta$ to $(\dag)$ that maximises $\textbf{A}\cdot z$.
    
    Let $w$ be the $n$-tuple obtained from $z$ by replacing $z_n$ with $\sigma^d z_n$ (respectively $\sigma^{-d}z_n$) if $A_n > 0$ (respectively $A_n< 0$). Since $z$ does not satisfy (i) or (ii), we have that $w\in \tilde{R}^n_\Delta$ and $u_n\leq w_n\leq v_n$. By Lemma \ref{ordering}, there is no $r\in \tilde{R}^n_\Delta$ such that $\textbf{A}\cdot r$ lies strictly between $\textbf{A}\cdot z$ and $\textbf{A}\cdot w$. Recalling that $\textbf{A}\cdot z<  \textbf{A}\cdot P_\Delta(x-y_1;\textbf{A})$, this shows that $\textbf{A}\cdot w\leq \textbf{A}\cdot P_\Delta(x-y_1;\textbf{A})$.
    
    By a similar argument, $\textbf{B}\cdot w\geq \textbf{B}\cdot P_\Delta(y_2-x;-\textbf{B})$. Thus,
    \[y_1+\textbf{A}\cdot w \leq \textbf{A}\cdot P_\Delta(x-y_1;\textbf{A}) < x < y_2 + \textbf{B}\cdot P_\Delta(y_2-x;-\textbf{B}) \leq y_2+\textbf{B}\cdot w,\]
    so $w$ is a witness to $(\dag)$. By Lemma \ref{ordering}, $\textbf{A}\cdot z<\textbf{A}\cdot w$, contradicting our choice of $z$.
\end{proof}
\begin{proof}[Proof of Proposition \ref{EnFnimpliesGn}]
    Let
    \[\mathcal{F}_{\textbf{A}, \textbf{B}}:=\{(\textbf{A},\textbf{B}),(-\textbf{B},-\textbf{A})\}\cup\{(\textbf{A},\textbf{F}^i), (-\textbf{B},\textbf{F}^i): 1\leq i\leq n\}\]
    if $\textbf{A}, \textbf{B}$ are tuples of non-zero operators, and let $\mathcal{F}_{\textbf{A}, \textbf{B}}:=\emptyset$ otherwise (recall that $\textbf{F}^i$ was defined as the $i^{\text{th}}$ standard tuple of operators). We claim that this witnesses the proposition.
    
    Let $\tilde{R}\subseteq^d R$ for some $d\in\N^+$, and let $\Delta\in d\N$ be sufficiently large as in Lemma \ref{canonicalwitness}. Let $\phi(x;y,u,v)$ be a $(G_n; \textbf{A}, \textbf{B}, \tilde{R}, \Delta)$ formula, say
    \[\phi(x;y,u,v)=T_{\tilde{R}}(u,v)\wedge \exists z\in \tilde{R}^n_\Delta\bigg(y_1+\textbf{A}\cdot z < x<y_2+\textbf{B}\cdot z\wedge \bigwedge_{i=1}^n u_i\leq z_i\leq v_i\bigg),\]
    where some of the $u_i$ (respectively $v_i$) may be $-\infty$ (respectively $+\infty$). Write $T(u,v)$ for $T_{\tilde{R}}(u,v)$.
    
    If $v_1=+\infty$ and $A_1=_R0<_RB_1$, then $\phi(x;y,u,v)$ is equivalent to
    \[\zeta(x;y,u,v):=T(u,v)\wedge \exists z\in \tilde{R}^n_\Delta \bigg(y_1+\textbf{A}\cdot z < x \wedge \bigwedge_{i=1}^n u_i\leq z_i\leq v_i\bigg).\]
    Indeed, clearly $\phi$ implies $\zeta$, and if $z\in \tilde{R}^n_\Delta$ witnesses $\zeta$, then for all/some sufficiently large $a\in \tilde{R}$, we have $w:=(a,z_{>1})\in \tilde{R}^n_\Delta$ and
    \[y_1+\textbf{A}\cdot w=y_1+\textbf{A}\cdot z<x<y_2+\textbf{B}\cdot w \wedge \bigwedge_{i=1}^n u_i\leq z_i\leq w_i <v_i.\]
    But $\zeta$ is equivalent to
    \[T(u,v)\wedge x>y_1+\inf\bigg\lbrace \textbf{A}\cdot z: z\in\tilde{R}^n_\Delta, \bigwedge_{i=1}^n u_i\leq z_i\leq v_i\bigg\rbrace,\]
    which is a basic formula. Thus, if $v_1=+\infty$ and $A_1=_R0<_RB_1$, then $\phi$ is equivalent to a basic formula. A similar situation arises if $v_1=+\infty$ and $A_1<_R0=_RB_1$, so henceforth suppose neither case holds. Let $\bar{\phi}(x;y,u,v,z)$ be the formula
    \[y_1+\textbf{A}\cdot z < x<y_2+\textbf{B}\cdot z \wedge \bigwedge_{i=1}^n u_i\leq z_i\leq v_i.\]
    For $1\leq i\leq n$, let
    \begin{align*}
        \alpha_i(x;y,u,v)&:=T(u,v)\wedge \exists z\in \tilde{R}^n_\Delta \left(z_i=\sigma^{\Delta}z_{i+1}\wedge \bar{\phi}(x;y,u,v,z)\right),\\
        \beta_i(x;y,u,v)&:=T(u,v)\wedge \exists z\in \tilde{R}^n_\Delta \left(z_i=u_i\wedge \bar{\phi}(x;y,u,v,z)\right),\\
        \gamma_i(x;y,u,v)&:=T(u,v)\wedge \exists z\in \tilde{R}^n_\Delta \left(z_i=v_i\wedge \bar{\phi}(x;y,u,v,z)\right),
    \end{align*}
    where $z_{n+1}:=\min\tilde{R}$. Furthermore, if $\textbf{A}$ and $\textbf{B}$ are tuples of non-zero operators then let
    \begin{align*}
        \theta(x;y,u,v):=T(u,v)&\wedge x-y_1>\inf \textbf{A}\cdot \tilde{R}^n_\Delta\\
        &\wedge x<y_2+\textbf{B}\cdot P_\Delta(x-y_1;\textbf{A})\wedge\bigwedge_{i=1}^n u_i\leq P^i_\Delta(x-y_1;\textbf{A})\leq v_i,\\
        \xi(x;y,u,v):=T(u,v)&\wedge y_2-x>\inf -\textbf{B}\cdot \tilde{R}^n_\Delta\\
        &\wedge x>y_1+\textbf{A}\cdot P_\Delta(y_2-x;-\textbf{B})\wedge\bigwedge_{i=1}^n u_i\leq P^i_\Delta(y_2-x;-\textbf{B})\leq v_i.
    \end{align*}
    By Lemma \ref{canonicalwitness} (and Lemma \ref{Plemma}), $\phi(x;y,u,v)$ is equivalent to
    \[\begin{cases}
        \theta\vee\xi\vee\bigvee_{i=1}^n (\alpha_i\vee \beta_i\vee \gamma_i) &\text{if }\textbf{A},\textbf{B} \text{ are tuples of non-zero operators},\\
        \bigvee_{i=1}^n (\alpha_i\vee \beta_i\vee \gamma_i) &\text{otherwise}.
    \end{cases}\]
    
    Observe that $\theta$ is a Boolean combinations of basic formulas and descendants of $(F_n;\textbf{I},\textbf{J},\tilde{R},\Delta)$ formulas for $(\textbf{I}, \textbf{J})\in \mathcal{F}_{\textbf{A}, \textbf{B}}$, since, for all $1\leq i\leq n$,
    \[u_i\leq P^i_\Delta(x-y_1;\textbf{A})\leq v_i\leftrightarrow u_i-1<\textbf{F}^i\cdot P_\Delta(x-y_1;\textbf{A})\wedge \neg (v_i<\textbf{F}^i\cdot P_\Delta(x-y_1;\textbf{A})).\]
    But this is also true for $\xi$, since, for example, $x>y_1+\textbf{A}\cdot P_\Delta(y_2-x;-\textbf{B})$ is a descendant of $-x>-y_2+\textbf{A}\cdot P_\Delta(-y_1+x;-\textbf{B})$, which is equivalent to $x-y_2<-\textbf{A}\cdot P_\Delta(x-y_1;-\textbf{B})$.
    
    For all $1\leq i\leq n$, $\alpha_i$, $\beta_i$, and $\gamma_i$ are equivalent to the conjunction of $T(u,v)$, which is a basic formula, and an $(E_{n-1},\tilde{R})$ formula, by substituting $z_i$ with $\sigma^\Delta z_{i+1}$, $u_i$, or $v_i$ as appropriate.
    
    Thus, $\phi$ is equivalent to a Boolean combination of basic formulas, $(E_{n-1};\tilde{R})$ formulas, and descendants of $(F_n;\textbf{I},\textbf{J},\tilde{R},\Delta)$ formulas for $(\textbf{I},\textbf{J})\in \mathcal{F}_{\textbf{A}, \textbf{B}}$.
\end{proof}

We are now ready to prove Theorem \ref{sufficiency}.
\begin{proof}[Proof of Theorem \ref{sufficiency}]
    Assume the criterion holds. By Proposition \ref{agenda}, it suffices to prove that every $(E_n)$ formula has a strong honest definition. We do so by induction on $n\in\N$. An $(E_0)$ formula is a basic formula, so suppose $n\geq 1$.

    Let $\phi$ be an $(E_n;\tilde{R})$ formula, where $\tilde{R}\subseteq^d R$ for some $d\in \N^+$. Let $\mathcal{F}_\phi$ be as in Corollary \ref{EntoFn}. Then, for all $\Delta\in d\N$ sufficiently large, $\phi$ is equivalent to a Boolean combination of basic formulas and descendants of $(E_{n-1};\tilde{R})$ or $(F_n; \textbf{I}, \textbf{J}, \tilde{R}, \Delta)$ formulas for $(\textbf{I},\textbf{J})\in \mathcal{F}_\phi$. By the induction hypothesis, every $(E_{n-1};\tilde{R})$ formula has a strong honest definition. By assumption, for all $\Delta\in d\N$ sufficiently large, every $(F_n; \textbf{I}, \textbf{J}, \tilde{R}, \Delta)$ formula for $(\textbf{I},\textbf{J})\in \mathcal{F}_\phi$ has a strong honest definition, noting that $\mathcal{F}_\phi$ is finite. Thus, $\phi$ is a Boolean combination of formulas with strong honest definitions.
\end{proof}
The rest of the paper is thus dedicated to establishing the sufficiency criterion in Theorem \ref{sufficiency}, by constructing strong honest definitions for $(F_n; \textbf{A}, \textbf{B}, \tilde{R}, \Delta)$ formulas with $\Delta$ sufficiently large. Note that this then gives a strong honest definition for \textit{every} $\mathcal{L}$-formula $\phi(x;y)$ with $\abs{x}=1$, since we have exhibited a way to write every such formula as a Boolean combination of basic formulas and descendants of $(F_n; \textbf{A}, \textbf{B}, \tilde{R}, \Delta)$ formulas with $\Delta$ sufficiently large. Indeed, by Proposition \ref{agenda}, every $\mathcal{L}$-formula $\phi(x;y)$ with $\abs{x}=1$ is equivalent to a Boolean combination of basic formulas and descendants of $(E_n)$ formulas. Example \ref{LOexample} gives strong honest definitions for basic formulas, and the proof of Corollary \ref{EntoFn} describes an algorithm for writing every $(E_n; \tilde{R})$ formula as a Boolean combination of descendants of $(E_{n-1}; \tilde{R})$ formulas and descendants of $(F_n; \textbf{A}, \textbf{B}, \tilde{R}, \Delta)$ formulas with $\Delta$ sufficiently large.
\section{Main construction}\label{sectionmain}
Recall that $R\subseteq\N$ is our fixed congruence-periodic sparse predicate. In this section, we show that every $(F_n;\textbf{A},\textbf{B},\tilde{R},\Delta)$ formula with $\Delta$ sufficiently large has a strong honest definition.

The broad strategy is to induct on $n$. Theorem \ref{catchall} can be seen as a stronger version of the $n=1$ case, and Theorem \ref{multidimension} handles the inductive step.

The following lemma transpires to be surprisingly useful.
\begin{lemma}\label{innocuous}
    Let $d,n\in\N^+$, $\tilde{R}\subseteq^d R$, and $\textbf{A}$ be an $n$-tuple of non-zero operators with $A_1>_R 0$ (respectively $A_1<_R 0$). Then there is $\Lambda\in\N$ such that the following holds.
    
    Let $\Delta\in d\N$ be sufficiently large, and let $s,t,x\in\Z$ be such that $s\leq t\leq x$ (respectively $s\geq t\geq x$). Then there is $0\leq\alpha\leq\Lambda$ such that $P^1_\Delta(x-s;\textbf{A})=\sigma^{\alpha} P^1_\Delta(x-t;\textbf{A})$ or $P^1_\Delta(x-s;\textbf{A})=\sigma^{\alpha} P^1_\Delta(t-s;\textbf{A})$.
\end{lemma}
    Let us give an intuitive interpretation of this lemma. Assuming $A_1>0$ for the purpose of this discussion, the lemma simply says that if $s\leq t\leq x$, then $x-s$ is `close' (with respect to the function $P^1_\Delta(\cdot;\textbf{A})$) to either $x-t$ or $t-s$.
\begin{proof}[Proof of Lemma \ref{innocuous}]
By Lemma \ref{beatsme}, we can fix $\Lambda\in\N$ such that $|A_1\sigma^\Lambda r|>|8A_1\sigma^d r|$ for all $r\in R$. Let $\Delta\in d\N$ be sufficiently large, and let $s,t,x\in\Z$ be such that $s\leq t\leq x$ if $A_1>_R 0$ and $s\geq t\geq x$ if $A_1<_R 0$. Let $w:=P_\Delta(x-t;\textbf{A})$ and $z:=P_\Delta(t-s;\textbf{A})$.

First suppose $A_1>_R 0$. Then
\[t-s\leq \textbf{A}\cdot Q_\Delta(t-s;\textbf{A})< 2A_1Q^1_\Delta(t-s;\textbf{A})\leq 2A_1 \sigma^d z_1,\]
where the first and last inequalities are by Lemma \ref{Plemma} and the second inequality is by Lemma \ref{domination}. Similarly, $x-t< 2A_1\sigma^d w_1$. But now
\[x-s=(x-t)+(t-s)< 4A_1\sigma^d \max\{z_1,w_1\}<\frac{1}{2}A_1 \sigma^\Lambda \max\{z_1,w_1\},\]
so, by Lemma \ref{domination}, $x\leq s+\textbf{A}\cdot u$ for all $u\in \tilde{R}^n_\Delta$ with $u_1\geq \sigma^\Lambda \max\{z_1,w_1\}$. Thus, $P^1_\Delta(x-s;\textbf{A})<\sigma^\Lambda \max\{z_1, w_1\}$. But $x\geq t\geq s$, so $\textbf{A}\cdot P_\Delta(x-s;\textbf{A})\geq \max\{\textbf{A}\cdot z, \textbf{A}\cdot w\}$, and thus $P^1_\Delta(x-s;\textbf{A})\geq \max\{z_1, w_1\}$ by Lemma \ref{ordering}.

Now suppose $A_1<_R 0$. Then $t-s>A_1z_1$ and $x-t>A_1w_1$ by Lemma \ref{Plemma}, whence
\[x-s=(x-t)+(t-s)>2A_1\max\{z_1,w_1\}>\frac{1}{4}A_1\sigma^\Lambda \max\{z_1,w_1\},\]
so, by Lemma \ref{domination}, $x> s+\textbf{A}\cdot u$ for all/some $u\in \tilde{R}^n_\Delta$ with $u_1= \sigma^\Lambda \max\{z_1,w_1\}$. Thus, $P^1_\Delta(x-s;\textbf{A})\leq \sigma^\Lambda \max\{z_1, w_1\}$. But $x\leq t\leq s$, so $\textbf{A}\cdot P_\Delta(x-s;\textbf{A})\leq \min\{\textbf{A}\cdot z, \textbf{A}\cdot w\}$, and thus $P^1_\Delta(x-s;\textbf{A})\geq \max\{z_1, w_1\}$ by Lemma \ref{ordering}.
\end{proof}
\begin{lemma}\label{firstcoorddef}
    Let $d,n\in\N^+$, $\tilde{R}\subseteq^d R$, $\textbf{A}$ be an $n$-tuple of non-zero operators, and $\Delta\in d\N$ be sufficiently large. Then the formula $\phi(x;y):=P^1_\Delta(x-y_1;\textbf{A})=y_2$ has a strong honest definition, given by the conjunction of strong honest definitions for the basic formulas
    \[\phi_1(x;y):=\begin{cases}
        x-y_1\leq \min\{\textbf{A}\cdot z: z\in \tilde{R}^n_\Delta, z_1=\sigma^d N\}&\text{if }A_1>_R0,\\
        x-y_1> \min\{\textbf{A}\cdot z: z\in \tilde{R}^n_\Delta, z_1=N\}&\text{if }A_1<_R0,
    \end{cases}\]
    where $N:=\sigma^{n\Delta}(\min \tilde{R})$, and
    \[\phi_2(x;y):=\min\{\textbf{A}\cdot z: z\in \tilde{R}^n_\Delta, z_1=y_2\}<x-y_1\leq \min\{\textbf{A}\cdot z: z\in \tilde{R}^n_\Delta, z_1=\sigma^{\varepsilon d} y_2\},\]
    where $\varepsilon:=1$ if $A_1>_R0$ and $\varepsilon:=-1$ if $A_1<_R0$.
\end{lemma}
\begin{proof}
Observe that
\[\phi(x;y)\leftrightarrow (y_2=N\wedge \phi_1(x;y))\vee (y_2\in\tilde{R}\wedge y_2>N\wedge \phi_2(x;y)).\]
Now apply Lemma \ref{boolcomb}.
\end{proof}

In the following theorem, we construct strong honest definitions for a class of formulas that includes all $(F_1;\textbf{A},\textbf{B},\tilde{R},\Delta)$ formulas with $\Delta$ sufficiently large (this inclusion is spelt out in Corollary \ref{onedimension}).
\begin{theorem}\label{catchall}
    Let $\theta(x;y)$ be a formula with $\abs{x}=1$, and suppose the formulas $\theta(x;y)$ and $\theta'(x;w,y):=\theta(x-w;y)$ both have strong honest definitions, where $\abs{w}=1$. Let $\gamma(x;y^{(1)},...,y^{(k)})$ be a strong honest definition for $\theta$.
    
    Let $d,n\in\N^+$, $\tilde{R}\subseteq^d R$, $\textbf{A}$ be an $n$-tuple of non-zero operators, and let $\Lambda\in\N$ be as in Lemma \ref{innocuous}. Let $\Delta\in d\N$ be sufficiently large, $t\in\Z$, and $f$ be an $\mathcal{L}$-definable function of arity 1. Then the formula
    \[\phi(x;w,y):=\theta\left(tx-f(P^1_\Delta(x-w;\textbf{A}));y\right)\]
    has a system of strong honest definitions $\{\zeta_{I_0J_0\cdots I_\Lambda J_\Lambda K}: I_\alpha\sqcup J_\alpha\subseteq [k]\text{ for all }0\leq \alpha\leq \Lambda, K\subseteq \{0, ..., \Lambda\}\}$, where $\zeta_{I_0J_0\cdots I_\Lambda J_\Lambda K}(x;...)$ is given by the conjunction of the following:
    \begin{enumerate}[(i)]
        \item A strong honest definition $\zeta_1(x;...)$ for the basic formula $\phi_1(x;w,y):=x\leq w$;
        \item A strong honest definition $\zeta_2(x;...)$ for the formula $\phi_2(x;w,y):=\theta(tx-f(\sigma^{n\Delta}(\min \tilde{R}));y)$, which exists since the formula is a descendant of $\theta$;
        \item For each $0\leq \alpha\leq \Lambda$, a strong honest definition $\zeta^\alpha_3(x;...)$ for the formula $\phi^\alpha_3(x;w,y,w',y'):=\theta'(tx;f(\sigma^\alpha P^1_\Delta(w'-w;\textbf{A})),y)$, which exists since the formula is a descendant of $\theta'$;
        \item For each $0\leq \alpha\leq \Lambda$, a strong honest definition $\zeta^\alpha_4(x;...)$ for the formula $\phi^\alpha_4(x;w,y,w',y'):=P^1_\Delta(x-w;\textbf{A})=\sigma^\alpha P^1_\Delta(w'-w;\textbf{A})$, which exists by Lemma \ref{firstcoorddef} (and Lemma \ref{descendants});
                \item For each $0\leq \alpha\leq \Lambda$, the formula
                \[\zeta^\alpha_{I_\alpha J_\alpha}(x;w,y^{(i)}: i\in [k]\setminus(I_\alpha\cup J_\alpha)):=\gamma(tx-f(\sigma^\alpha P^1_\Delta(x-w;\textbf{A}));\hat{y}^{(1)}, ..., \hat{y}^{(k)}),\]
                where, for $1\leq i\leq k$,
        \[\hat{y}^{(i)}:=\begin{cases}
            (0, ..., 0)&\text{if }i\in I_\alpha,\\
            (1, ..., 1)&\text{if }i\in J_\alpha,\\
            y_i&\text{otherwise};
        \end{cases}\]
        and
        \item The formula 
        \[\bigwedge_{\alpha\in K}P^1_\Delta(x-w_\alpha;\textbf{A})=P^1_\Delta(x-w'_\alpha;\textbf{A})=\sigma^\alpha P^1_\Delta(x-w''_\alpha;\textbf{A}).\]
    \end{enumerate}
\end{theorem}
Let us first describe the idea of the proof, assuming $A_1>_R 0$ for the purpose of this discussion. We wish to replace $P^1_\Delta(x-w;\textbf{A})$ in $\phi(x;w,y)$ with a more tractable expression; we can do so by Lemma \ref{innocuous}, which gives us $\Lambda\in\N$ satisfying the following.

Let $x_0\in\Z$ and $S\subseteq \Z^{1+\abs{y}}$ with $2\leq\abs{S}<\infty$. Here and henceforth, when it is written that $(b,a)\in S$, it is understood that $\abs{b}=1$ and $\abs{a}=\abs{y}$. Let $u:=\max(\{b: (b,a)\in S, x_0>b\}\cup\{\min_{(b,a)\in S}b\})$. For all $(b,a)\in S$, if $b>u$ then $P^1_\Delta(x-b;\textbf{A})=\sigma^{n\Delta}(\min \tilde{R})$, and if $b\leq u$ then either
\begin{enumerate}[(i)]
    \item $P^1_\Delta(x-b;\textbf{A})=\sigma^\alpha P^1_\Delta(x-u;\textbf{A})$ for some $0\leq \alpha\leq \Lambda$; or
    \item $P^1_\Delta(x-b;\textbf{A})=\sigma^\alpha P^1_\Delta(u-b;\textbf{A})$ for some $0\leq \alpha\leq \Lambda$.
\end{enumerate}
In each of these cases, replacing $P^1_\Delta(x-b;\textbf{A})$ with the respective expression gives a formula for which we have strong honest definitions. 
\begin{proof}[Proof of Theorem \ref{catchall}]
    By Lemma \ref{domination}, we may assume $\Delta\in d\N$ is sufficiently large that $\min \textbf{A}\cdot \tilde{R}^n_\Delta>0$ if $A_1>_R 0$ and $\max \textbf{A}\cdot \tilde{R}^n_\Delta<0$ if $A_1<_R 0$.
    
    Fix $x_0\in\Z$ and $S\subseteq \Z^{1+\abs{y}}$ with $2\leq \abs{S}<\infty$. Write $\pi_1(S):=\{b: (b,a)\in S\}$ and $\pi_2(S):=\{a: (b,a)\in S\}$. Let $(b_0,a_0)\in S$ be such that
    \[b_0=\begin{cases}
        \min \pi_1(S)&\text{if }A_1>_R0,\\
        \max \pi_1(S)&\text{if }A_1<_R0.
    \end{cases}\]
    Define
    \[u:=\begin{cases}
        \max(\{b\in \pi_1(S): x_0>b\}\cup\{b_0\})&\text{if }A_1>_R0,\\
        \min(\{b\in \pi_1(S): x_0\leq b\}\cup\{b_0\})&\text{if }A_1<_R0.
    \end{cases}\]
    
    For $i\in \{1,2\}$, let $c_i\in S^{<\omega}$ be such that $x_0\models \zeta_i(x; c_i)$ and $\zeta_i(x; c_i)\vdash \text{tp}_{\phi_i}(x_0/S)$. For $i\in \{3,4\}$ and $0\leq \alpha\leq \Lambda$, let $c_i^\alpha\in (S^2)^{<\omega}$ be such that $x_0\models \zeta_i^\alpha(x; c_i^\alpha)$ and $\zeta_i^\alpha(x; c_i^\alpha)\vdash \text{tp}_{\phi_i^\alpha}(x_0/S^2)$. 

    Let $T:=\pi_2(S)\cup\{(0, ..., 0), (1, ..., 1)\}\subseteq \Z^y$. Then $\abs{T}\geq 2$, so for $0\leq \alpha\leq \Lambda$, there is $e^\alpha\in T^k$ such that $tx_0-f(\sigma^\alpha P^1_\Delta(x_0-u;\textbf{A}))\models \gamma(x;e^\alpha)$ and $\gamma(x;e^\alpha)\vdash \text{tp}_\theta(tx_0-f(\sigma^\alpha P^1_\Delta(x_0-u;\textbf{A}))/T)$. There are disjoint $I_\alpha,J_\alpha\subseteq [k]$ and $c^\alpha\in \pi_2(S)^{<\omega}$ such that
    \[\gamma(tx-f(\sigma^\alpha P^1_\Delta(x-u;\textbf{A}));e^\alpha)=\zeta^\alpha_{I_\alpha J_\alpha}(x;u,c^\alpha),\]
    whence $x_0\models \zeta^\alpha_{I_\alpha J_\alpha}(x;u,c^\alpha)$.
    
    For $0\leq\alpha\leq\Lambda$, let $S_\alpha:=\left\lbrace b\in\pi_1(S): P^1_\Delta(x_0-b;\textbf{A})=\sigma^\alpha P^1_\Delta(x_0-u;\textbf{A})\right\rbrace\subseteq\Z$, and if $S_\alpha\neq \emptyset$, let $l^{(\alpha)}:=\min S_\alpha$ and $r^{(\alpha)}:=\max S_\alpha$.
    
    Then we have that
    \begin{align*}
        x_0\models \bigwedge_{i=1}^2\zeta_i(x; c_i)&\wedge \bigwedge_{i=3}^4\bigwedge_{\alpha=0}^\Lambda \zeta_i^\alpha(x; c_i^\alpha)\wedge \bigwedge_{\alpha=0}^\Lambda \zeta^\alpha_{I_\alpha J_\alpha}(x;u,c^\alpha)\\
        &\wedge \bigwedge_{\substack{\alpha=0\\S_\alpha\neq\emptyset}}^\Lambda P^1_\Delta(x-l^{(\alpha)};\textbf{A})=P^1_\Delta(x-r^{(\alpha)};\textbf{A})=\sigma^\alpha P^1_\Delta(x-u;\textbf{A}),
    \end{align*}
    and we claim that this formula, which is an instance of $\zeta_{I_0J_0\cdots I_\Lambda J_\Lambda K}$ for $K:=\{0\leq\alpha\leq\Lambda: S_\alpha\neq\emptyset\}$, entails $\text{tp}_{\phi}(x_0/S)$.

    Indeed, suppose $x_1\in\Z$ satisfies this formula, and let $(b',a')\in S$. We wish to show that $\phi(x_0;b',a')$ if and only if $\phi(x_1;b',a')$. Since $x_0,x_1\models \zeta_1(x; c_1)$, we have that for $i\in\{0,1\}$,
    \[u=\begin{cases}
        \max(\{b\in \pi_1(S): x_i>b\}\cup\{b_0\})&\text{if }A_1>_R0,\\
        \min(\{b\in \pi_1(S): x_i\leq b\}\cup\{b_0\})&\text{if }A_1<_R0.
    \end{cases}\]
    
    Suppose $b'>u$ and $A_1>_R 0$. Then, for $i\in\{0,1\}$, we have $x_i-b'\leq 0<\min \textbf{A}\cdot \tilde{R}^n_\Delta$ and so $P^1_\Delta(x_i-b';\textbf{A})=\sigma^{n\Delta}(\min\tilde{R})$ by Remark \ref{Premark}. Thus, for $i\in\{0,1\}$, we have $\phi(x_i;b',a')\Leftrightarrow \phi_2(x_i;b',a')$. But now, since $x_0,x_1\models \zeta_2(x;c_2)$, we have $\phi_2(x_0;b',a')\Leftrightarrow \phi_2(x_1;b',a')$, whence $\phi(x_0;b',a')\Leftrightarrow \phi(x_1;b',a')$.

    The case where $b'<u$ and $A_1<_R 0$ is similar, so henceforth suppose either ($b'\leq u$ and $A_1>_R 0$) or ($b'\geq u$ and $A_1<_R 0$). By Lemma \ref{innocuous}, we have either
    \begin{enumerate}[(i)]
        \item That $P^1_\Delta(x_0-b';\textbf{A})=\sigma^\alpha P^1_\Delta(x_0-u;\textbf{A})$ for some $0\leq \alpha\leq \Lambda$; or
        \item That $P^1_\Delta(x_0-b';\textbf{A})=\sigma^\alpha P^1_\Delta(u-b';\textbf{A})$ for some $0\leq \alpha\leq \Lambda$.
    \end{enumerate}
    
    If $0\leq \alpha\leq \Lambda$ is such that $P^1_\Delta(x_0-b';\textbf{A})=\sigma^\alpha P^1_\Delta(u-b';\textbf{A})$, then since $x_0,x_1\models \zeta_4^\alpha(x;c_4^\alpha)$, we have $P^1_\Delta(x_0-b';\textbf{A})=P^1_\Delta(x_1-b';\textbf{A})=\sigma^\alpha P^1_\Delta(u-b';\textbf{A})$. Thus, for $i\in\{0,1\}$, we have
    \[\phi(x_i;b',a')\Leftrightarrow \theta\left(tx_i-f(\sigma^\alpha P^1_\Delta(u-b';\textbf{A}));a'\right),\]
    and so
    \[\phi(x_i;b',a')\Leftrightarrow\theta'\left(tx_i;f(\sigma^\alpha P^1_\Delta(u-b';\textbf{A})),a'\right).\]
    But now, since $x_0,x_1\models \zeta_3^\alpha(x;c_3^\alpha)$, we have
    \[\theta'\left(tx_0;f(\sigma^\alpha P^1_\Delta(u-b';\textbf{A})),a'\right)\Leftrightarrow\theta'\left(tx_1;f(\sigma^\alpha P^1_\Delta(u-b';\textbf{A})),a'\right),\]
    whence $\phi(x_0;b',a')\Leftrightarrow \phi(x_1;b',a')$.
    
    Suppose instead that we have $P^1_\Delta(x_0-b';\textbf{A})=\sigma^\alpha P^1_\Delta(x_0-u;\textbf{A})$ for some $0\leq\alpha\leq \Lambda$, and so $l^{(\alpha)}\leq b'\leq r^{(\alpha)}$. But now
    \[x_0,x_1\models P^1_\Delta(x-l^{(\alpha)};\textbf{A})=P^1_\Delta(x-r^{(\alpha)};\textbf{A})=\sigma^\alpha P^1_\Delta(x-u;\textbf{A}),\]
    so by Lemma \ref{ordering} we must have $x_0,x_1\models P^1_\Delta(x-b';\textbf{A})=\sigma^\alpha P^1_\Delta(x-u;\textbf{A})$. Thus, for $i\in\{0,1\}$,
    \[\phi(x_i;b',a')\Leftrightarrow \theta\left(tx_i-f(\sigma^\alpha P^1_\Delta(x_i-u;\textbf{A}));a'\right).\]
    But now, since $x_0,x_1\models \zeta_{I_\alpha J_\alpha}^\alpha(x;u,c^\alpha)$, we have $x_0,x_1\models\gamma(tx-f(\sigma^\alpha P^1_\Delta(x-u;\textbf{A}));e^\alpha)$ and so
    \[\theta\left(tx_0-f(\sigma^\alpha P^1_\Delta(x_0-u;\textbf{A}));a'\right)\Leftrightarrow \theta\left(tx_1-f(\sigma^\alpha P^1_\Delta(x_1-u;\textbf{A}));a'\right),\]
    whence $\phi(x_0;b',a')\Leftrightarrow \phi(x_1;b',a')$, which finishes the proof.
\end{proof}
\begin{cor}\label{onedimension}
    Let $\tilde{R}\subseteq^d R$ for some $d\in\N^+$. Let $t\in\Z$, $\textbf{A}$ be a tuple of non-zero operators, $f$ be an $\mathcal{L}$-definable function of arity 1, and $\square\in\{<,>\}$. Let $\Delta\in d\N$ be sufficiently large. Then the formula $\phi(x;y):=tx-y_2\;\square\; f(P^1_\Delta(x-y_1;\textbf{A}))$ has a strong honest definition. In particular, given operators $A,B$ with $A\neq_R 0$, every $(F_1; A,B,\Delta,\tilde{R})$ formula with $\Delta\in d\N$ sufficiently large has a strong honest definition.
\end{cor}
\begin{proof}
    This follows directly from Theorem \ref{catchall} since, for $\theta(x;y_2):=x\;\square\;y_2$,
    \[\phi(x;y)=\theta\left(tx-f(P^1_\Delta(x-y_1;\textbf{A}));y_2\right),\]
    and the formulas $\theta(x;y_2)$ and $\theta'(x;w,y_2):=\theta(x-w;y_2)$ have strong honest definitions by Example \ref{LOexample}.
\end{proof}
Recall that, given an $n$-tuple $\nu=(\nu_1,...,\nu_n)$, we let $\nu_{>1}$ denote $(\nu_2,...,\nu_n)$.
\begin{lemma}\label{reduction}
    Let $d,n\in\N^+$ with $n\geq 2$, $\tilde{R}\subseteq^d R$, $\textbf{A}$ be an $n$-tuple of non-zero operators, and $\Delta\in d\N$ be sufficiently large. Let $a\in\Z$ be such that
    \[a>\inf \textbf{A}\cdot \tilde{R}^n_\Delta\wedge a\leq \max \{\textbf{A}\cdot z: z\in \tilde{R}^n_\Delta, z_1=P^1_\Delta(a;\textbf{A})\}.\]
    Then $P_\Delta\left(a-A_1 P^1_\Delta(a;\textbf{A});\textbf{A}_{>1}\right)=P^{>1}_\Delta(a;\textbf{A})$.
\end{lemma}
\begin{proof}
    Let $u=P_\Delta(a;\textbf{A})$. Then
    \[\textbf{A}_{>1}\cdot u_{>1}+A_1 P^1_\Delta(a;\textbf{A})=\textbf{A}\cdot P_\Delta(a;\textbf{A})<a,\]
    and so $\textbf{A}_{>1}\cdot u_{>1}<a-A_1P^1_\Delta(a;\textbf{A})$. Thus, to show that $u_{>1}=P_\Delta\left(a-A_1 P^1_\Delta(a;\textbf{A});\textbf{A}_{>1}\right)$, it suffices to show that there is no $w\in \tilde{R}^{n-1}_\Delta$ such that $\textbf{A}_{>1}\cdot u_{>1}<\textbf{A}_{>1}\cdot w<a-A_1 P^1_\Delta(a;\textbf{A})$.

    Towards a contradiction, suppose such a $w\in \tilde{R}^{n-1}_\Delta$ existed, so $\textbf{A}\cdot u<\textbf{A}\cdot (P^1_\Delta(a;\textbf{A}),w)<a$. By definition of $u=P_\Delta(a;\textbf{A})$, we must have that $(P^1_\Delta(a;\textbf{A}),w)\not\in\tilde{R}^n_\Delta$, and so $w_1>\sigma^{-\Delta} P^1_\Delta(a;\textbf{A})$.
    
    Recalling the relevant notation from Definition \ref{defnmM}, let $v:=\max_{\textbf{A}} \{z\in \tilde{R}^n_\Delta: z_1=P^1_\Delta(a;\textbf{A})\}$, so by assumption, $\textbf{A}\cdot u<\textbf{A}\cdot (P^1_\Delta(a;\textbf{A}),w)<a\leq \textbf{A}\cdot v$. But now, since $u_1=v_1=P^1_\Delta(a;\textbf{A})$, we have
    \[\textbf{A}_{>1}\cdot u_{>1}<\textbf{A}_{>1}\cdot w<\textbf{A}_{>1}\cdot v_{>1},\]
    so by Lemma \ref{ordering}, $w_1\leq \max\{u_2, v_2\}\leq \sigma^{-\Delta}\max\{u_1, v_1\}=\sigma^{-\Delta} P^1_\Delta(a;\textbf{A})$, a contradiction.
\end{proof}
The following theorem describes how a strong honest definition for a $(F_n; ...)$ formula can be obtained from one for a $(F_{n-1}; ...)$ formula.
\begin{theorem}\label{multidimension}
    Let $d,n\in\N^+$ with $n\geq 2$, $\tilde{R}\subseteq^d R$, $\textbf{A}$ be an $n$-tuple of non-zero operators, and $\textbf{B}$ be an $n$-tuple of operators. Let $t\in\{0,1\}$ with $t=1$ unless $\textbf{B}=\textbf{F}^i$ for some $1\leq i\leq n$. Suppose that, for all $\Delta\in d\N$ sufficiently large, the formula
    \[\theta(x;y_1, y_2):=tx-y_2< \textbf{B}_{>1}\cdot P_\Delta(x-y_1;\textbf{A}_{>1},\tilde{R})\]
    has a strong honest definition. Then, for all $\Delta\in d\N$ sufficiently large, the formula
    \[\phi(x;y_1, y_2):=tx-y_2< \textbf{B}\cdot P_\Delta(x-y_1;\textbf{A},\tilde{R})\]
    has a strong honest definition, given by a conjunction of copies of strong honest definitions for
    \[\begin{cases}
        \phi_0&\text{if }\textbf{B}=\textbf{F}^1,\\
        \phi_1, \phi_2, \phi_3, \phi_4, \phi_5, (\phi_6^\alpha, \phi_7^\alpha: 0\leq \alpha\leq \Lambda)&\text{if }\textbf{B}\neq\textbf{F}^1, A_1\neq_R B_1,\text{ and }t=1,\\
        \phi_1, \phi_2, \phi_3, \phi_4, \phi_8&\text{otherwise},
    \end{cases}\]
    where
    \begin{align*}
        \phi_0(x;y_1,y_2)&:=tx-y_2<P^1_\Delta(x-y_1;\textbf{A}),\\
        \phi_1(x;y_1,y_2)&:=\inf \textbf{A}\cdot \tilde{R}^n_\Delta<x-y_1,\\
        \phi_2(x;y_1,y_2)&:=tx-y_2<\textbf{B}\cdot (\sigma^{(n-i)\Delta}(\min \tilde{R}): 0\leq i< n),\\
        \phi_3(x;y_1,y_2)&:=x-y_1>\max\{\textbf{A}\cdot z: z\in \tilde{R}^n_\Delta, z_1=P^1_\Delta(x-y_1;\textbf{A})\},\\
        \phi_4(x;y_1,y_2)&:=tx-y_2 < \textbf{B}\cdot \max_{\textbf{A}}\{ z\in \tilde{R}^n_\Delta: z_1=P^1_\Delta(x-y_1;\textbf{A})\},\\
        \phi_5(x;y_1,y_2)&:=\begin{cases}
        P^1_\Delta(x-y_1;\textbf{A})>\sigma^{\Delta}P_\Delta(y_1-y_2;B_1-A_1)&\text{if }B_1-A_1>_R0,\\
        P^1_\Delta(x-y_1;\textbf{A})<\sigma^{-\Delta}P_\Delta(y_1-y_2;B_1-A_1)&\text{if }B_1-A_1<_R0,
        \end{cases}\\
        \phi^\alpha_6(x;y_1,y_2)&:=P^1_\Delta(x-y_1;\textbf{A})=\sigma^\alpha P_\Delta(y_1-y_2;B_1-A_1),\\
        \phi_7^\alpha(x;y_1,y_2)&:=tx-y_2-B_1 \sigma^\alpha P_\Delta(y_1-y_2;B_1-A_1)\\
        &\;\;\;\;\;\;\;\;\;\;\;\;\;\;\;\;\;\;\;\;\;\;\;< \textbf{B}_{>1}\cdot P_\Delta(x-y_1-A_1\sigma^\alpha P_\Delta(y_1-y_2;B_1-A_1);\textbf{A}_{>1}),\\
        \phi_8(x;y_1,y_2)&:=\theta(x-A_1P^1_\Delta(x-y_1;\textbf{A});y_1,y_2).
    \end{align*}
\end{theorem}
From this we immediately obtain the sufficiency criterion in Theorem \ref{sufficiency} as a corollary.
\begin{cor}\label{sufficiencycriterion}
    Let $d,n\in\N^+$, $\tilde{R}\subseteq^d R$, $\textbf{A}$ be an $n$-tuple of non-zero operators, and $\textbf{B}$ be an $n$-tuple of operators. Then, for all sufficiently large $\Delta\in d\N$, every $(F_n;\textbf{A},\textbf{B},\tilde{R},\Delta)$ formula has a strong honest definition.
\end{cor}
\begin{proof}
    Induct on $n\in\N^+$, with Corollary \ref{onedimension} as the base case $n=1$ and Theorem \ref{multidimension} as the inductive step.
\end{proof}
Before proving Theorem \ref{multidimension}, let us first justify that the formulas $\phi_0, ..., \phi_8$ indeed have strong honest definitions, assuming $\Delta\in d\N$ is sufficiently large.

The formulas $\phi_0$, $\phi_3$, $\phi_4$, and $\phi_5$ have strong honest definitions by Corollary \ref{onedimension} and Lemma \ref{descendants}, applied with $\Delta\in d\N$ sufficiently large. As an example, to show that $\phi_3$ has a strong honest definition (assuming $\Delta\in d\N$ is sufficiently large), one applies Corollary \ref{onedimension} with $t=1$, $\square$ as $>$, and $f$ mapping $u\mapsto\max\{\textbf{A}\cdot z: z\in \tilde{R}^n_\Delta, z_1=u\}$ if $u\in R$ and $u\mapsto 0$ otherwise.

The formulas $\phi_1$ and $\phi_2$ are basic formulas, so have strong honest definitions.

For $0\leq\alpha\leq\Lambda$, the formula $\phi_6^\alpha$ has a strong honest definition by Lemmas \ref{firstcoorddef} and \ref{descendants}, since it is a descendant of the formula $P^1_\Delta(x-y_1;\textbf{A})=y_2$.

For $0\leq\alpha\leq\Lambda$, the formula $\phi_7^\alpha$ has a strong honest definition by Lemma \ref{descendants}, since it is a descendant of the formula $\theta(x;y_1,y_2)$, which is assumed to have a strong honest definition.

Finally, consider the formula $\phi_8$. It is a descendant of the formula
\[\phi'_8(x;w,y_1,y_2):=\theta(x-A_1P^1_\Delta(x-w;\textbf{A});y_1,y_2),\]
so by Lemma \ref{descendants} it suffices to show that $\phi'_8$ has a strong honest definition. Now the formula $\theta'(x;w,y_1,y_2):=\theta(x-w;y_1,y_2)$ is easily seen to be a descendant of $\theta$, which is assumed to have a strong honest definition, and hence so does $\theta'$ by Lemma \ref{descendants}. Thus, the formula $\phi'_8$ has a strong honest definition by Theorem \ref{catchall}, applied with $t=1$ and $f$ mapping $u\mapsto A_1u$ if $u\in R$ and $u\mapsto 0$ otherwise.

Thus, Theorem \ref{multidimension} is well-formulated; let us prove it.
\begin{proof}[Proof of Theorem \ref{multidimension}]
    Let $\Delta\in d\N$ be sufficiently large such that the function $z\mapsto \textbf{A}\cdot z$ is injective on $R^n_\Delta$, $\theta(x;y_1,y_2)$ has a strong honest definition, and all the strong honest definitions exist that are claimed to exist in the statement of the theorem. We will show that $\phi(x;y_1,y_2)$ is a Boolean combination of copies of
    \[\begin{cases}
        \phi_0&\text{if }\textbf{B}=\textbf{F}^1,\\
        \phi_1, \phi_2, \phi_3, \phi_4, \phi_5, (\phi_6^\alpha, \phi_7^\alpha: 0\leq \alpha\leq \Lambda)&\text{if }\textbf{B}\neq\textbf{F}^1, A_1\neq_R B_1,\text{ and }t=1,\\
        \phi_1, \phi_2, \phi_3, \phi_4, \phi_8&\text{otherwise},
    \end{cases}\]
    which suffices by Lemma \ref{boolcomb}.

    If $\textbf{B}=\textbf{F}^1$ then $\phi(x;y)\leftrightarrow \phi_0(x;y)$, so henceforth suppose $\textbf{B}\neq\textbf{F}^1$.
    
    If $\neg \phi_1(x;y)$ holds then $x-y_1\leq \inf \textbf{A}\cdot \tilde{R}^n_\Delta$, so by Remark \ref{Premark} we have $\phi(x;y)\leftrightarrow \phi_2(x;y)$. Henceforth condition on $\phi_1(x;y)$, whence by Lemma \ref{Plemma},
    \begin{equation}\tag{$\ast$}\label{eqnPQ}
        \textbf{A}\cdot P_\Delta(x-y_1;\textbf{A})<x-y_1.
    \end{equation} 

    If $\phi_3(x;y)$ holds, then $P_\Delta(x-y_1;\textbf{A})=\max_{\textbf{A}}\{ z\in \tilde{R}^n_\Delta: z_1=P^1_\Delta(x-y_1;\textbf{A})\}$ and so $\phi(x;y)\leftrightarrow \phi_4(x;y)$. Henceforth condition on $\neg\phi_3(x;y)$. Note then that, assuming $\Delta\in d\N$ is sufficiently large, Lemma \ref{reduction} implies
    \begin{equation}\tag{$\ast\ast$}\label{eqninductive}
        P^{>1}_\Delta(x-y_1;\textbf{A})=P_\Delta\left(x-y_1-A_1P^1_\Delta(x-y_1;\textbf{A});\textbf{A}_{>1}\right).
    \end{equation}    
    We now split into two cases: $A_1\neq _R B_1\wedge t=1$, and $(A_1=_R B_1 \wedge t=1)\vee(\textbf{B}=\textbf{F}^i \wedge t=0)$.\\
    
    \underline{Case 1: $A_1\neq_R B_1\wedge t=1$.} We will show that
    \[\phi(x;y)\leftrightarrow \phi_5(x;y)\vee\bigvee_{\alpha=-\Delta}^\Delta (\phi_6^\alpha(x;y)\wedge\phi_7^\alpha(x;y)).\]
    
    Let $\varepsilon:=1$ if $B_1-A_1>_R 0$, and $\varepsilon:=-1$ if $B_1-A_1<_R 0$.
    
    Firstly, suppose $\phi_\bot(x;y)$ holds, where
    \[\phi_\bot(x;y):=\begin{cases}
        P^1_\Delta(x-y_1;\textbf{A})<\sigma^{-\Delta}P_\Delta(y_1-y_2;B_1-A_1)&\text{if }B_1-A_1>_R0,\\
        P^1_\Delta(x-y_1;\textbf{A})>\sigma^{\Delta}P_\Delta(y_1-y_2;B_1-A_1)&\text{if }B_1-A_1<_R0.
    \end{cases}\]
    In particular, $\inf (B_1-A_1)\tilde{R}^1_\Delta < y_1-y_2$ by Remark \ref{Premark}, whence, for $\Delta\in d\N$ sufficiently large,
    \begin{alignat*}{2}
        y_1-y_2&>(B_1-A_1)P_\Delta(y_1-y_2;B_1-A_1)&&\;\;\;\;\;\text{by Lemma \ref{Plemma}}\\
        &>(B_1-A_1)\sigma^{\varepsilon\Delta}  P^1_\Delta(x-y_1;\textbf{A})&&\;\;\;\;\;\text{by }\phi_\bot(x;y)\\
        &>2^{\varepsilon}(B_1-A_1) P^1_\Delta(x-y_1;\textbf{A})&&\;\;\;\;\;\text{by Lemma \ref{beatsme}}\\
        &>(\textbf{B}-\textbf{A})\cdot P_\Delta(x-y_1;\textbf{A})&&\;\;\;\;\;\text{by Lemma \ref{domination}}.
    \end{alignat*}
    Thus,
    \begin{alignat*}{2}
        x-y_2&=x-y_1+y_1-y_2\\
        &>\textbf{A}\cdot P_\Delta(x-y_1;\textbf{A})+y_1-y_2&&\;\;\;\;\;\text{by (\ref{eqnPQ})}\\
        &>\textbf{A}\cdot P_\Delta(x-y_1;\textbf{A})+(\textbf{B}-\textbf{A})\cdot P_\Delta(x-y_1;\textbf{A})\\
        &=\textbf{B}\cdot P_\Delta(x-y_1;\textbf{A}),
    \end{alignat*}
    and so $\phi(x;y)\leftrightarrow \bot$.
    
    Next, suppose $\phi_5(x;y)$ holds. In particular, $\sup (B_1-A_1)\tilde{R}^1_\Delta \geq y_1-y_2$ by Remark \ref{Premark}, whence, for $\Delta\in d\N$ sufficiently large,
    \begin{alignat*}{2}
        y_1-y_2&\leq (B_1-A_1) Q_\Delta(y_1-y_2;B_1-A_1)&&\;\;\;\;\;\text{by Lemma \ref{Plemma}}\\
        &\leq (B_1-A_1) \sigma^{\varepsilon d} P_\Delta(y_1-y_2;B_1-A_1)&&\;\;\;\;\;\text{by Lemma \ref{Plemma}}\\
        &<(B_1-A_1) \sigma^{\varepsilon(d-\Delta)} P^1_\Delta(x-y_1;\textbf{A})&&\;\;\;\;\;\text{by }\phi_5(x;y).
    \end{alignat*}
    Using $\neg \phi_3(x;y)$, we have
    \[x-y_1\leq \max\{\textbf{A}\cdot z: z\in \tilde{R}^n_\Delta, z_1=P^1_\Delta(x-y_1;\textbf{A})\} < (A_1+\sigma^{-\lfloor \Delta/2\rfloor})P^1_\Delta(x-y_1;\textbf{A})\]
    by Lemma \ref{beatsme} (for $\Delta\in d\N$ sufficiently large). Thus,
    \begin{alignat*}{2}
        x-y_2&=x-y_1+y_1-y_2\\
        &< (A_1+\sigma^{-\lfloor \Delta/2\rfloor})P^1_\Delta(x-y_1;\textbf{A})+(B_1-A_1) \sigma^{\varepsilon(d-\Delta)} P^1_\Delta(x-y_1;\textbf{A})\\
        &< (B_1-\sigma^{-\lfloor \Delta/2\rfloor})P^1_\Delta(x-y_1;\textbf{A})&&\;\;\;\;\;\text{by Lemma \ref{beatsme}}\\
        &< \textbf{B}\cdot P_\Delta(x-y_1;\textbf{A})&&\;\;\;\;\;\text{by Lemma \ref{beatsme}},
    \end{alignat*}
    and so $\phi(x;y)\leftrightarrow \top$.
        
    
    Finally, suppose neither $\phi_\bot(x;y)$ nor $\phi_5(x;y)$ holds. Then there is $-\Delta\leq\alpha\leq\Delta$ such that $\phi_6^\alpha(x;y)$ holds. Conditioning on such $\phi_6^\alpha(x;y)$, we have
    \begin{align*}
        \phi(x;y)&\leftrightarrow x-y_2-B_1 P^1_\Delta(x-y_1;\textbf{A}) < \textbf{B}_{>1}\cdot P^{>1}_\Delta(x-y_1;\textbf{A})\\
        &\leftrightarrow x-y_2-B_1 P^1_\Delta(x-y_1;\textbf{A}) < \textbf{B}_{>1}\cdot P_\Delta\left(x-y_1-A_1P^1_\Delta(x-y_1;\textbf{A});\textbf{A}_{>1}\right)
    \end{align*}
    by (\ref{eqninductive}). But this is equivalent to $\phi^\alpha_7(x;y)$, since $\phi^\alpha_6(x;y)$ holds.\\
    
    \underline{Case 2: ($A_1=_R B_1 \wedge t=1$) or $(\textbf{B}=\textbf{F}^i \wedge t=0)$.} Recall that we have assumed $\textbf{B}\neq \textbf{F}^1$; note then in particular that $B_1=_R tA_1$. We have
    \begin{align*}
        \phi(x;y)&\leftrightarrow tx-B_1 P^1_\Delta(x-y_1;\textbf{A})-y_2<\textbf{B}_{>1}\cdot P^{>1}_\Delta(x-y_1;\textbf{A})\\
        &\leftrightarrow tx-B_1 P^1_\Delta(x-y_1;\textbf{A})-y_2<\textbf{B}_{>1}\cdot P_\Delta\left(x-y_1-A_1P^1_\Delta(x-y_1;\textbf{A});\textbf{A}_{>1}\right)
    \end{align*}
    by (\ref{eqninductive}). But this is equivalent to $\phi_8(x;y)$, since $tx-B_1 P^1_\Delta(x-y_1;\textbf{A})=t(x-A_1P^1_\Delta(x-y_1;\textbf{A}))$ by the fact that $B_1=tA_1$.
    \end{proof}
\begin{theorem}\label{final}
    The structure $(\Z,<,+,R)$ is distal.
\end{theorem}
\begin{proof}
Combine Theorem \ref{sufficiency} and Corollary \ref{sufficiencycriterion}.
\end{proof}
\bibliographystyle{plain}
\bibliography{divergebib}
\end{document}